\documentclass[11pt,a4paper,twoside]{amsart}

\usepackage[dvipsnames,svgnames,x11names,hyperref]{xcolor}
\usepackage[T1]{fontenc}
\usepackage[english]{babel}
\usepackage{pifont,fancybox,pstricks,pst-grad,pst-node}
\usepackage{amssymb,amscd,amstext,latexsym,array,verbatim,txfonts,url,graphicx,array,colortbl}
\usepackage{mathtools,microtype}
\usepackage{tikz,tikz-cd}
\usetikzlibrary{matrix,calc,positioning,arrows,decorations.pathreplacing,patterns,arrows}
\usepackage[pagebackref,colorlinks,citecolor=blue,linkcolor=blue,urlcolor=blue,filecolor=blue]{hyperref}
\usepackage{booktabs}
\renewcommand{\arraystretch}{1.2}

\usepackage{caption}
\renewcommand{\le}{\leqslant}
\renewcommand{\ge}{\geqslant}

\theoremstyle{plain}
\newtheorem{thm}{\bfseries Theorem}[section]
\newtheorem{conj}{\bfseries Conjecture}[section]
\newtheorem{lem}[thm]{\bfseries Lemma}
\newtheorem{prop}[thm]{\bfseries Proposition}
\newtheorem{cor}[thm]{\bfseries Corollary}
\theoremstyle{definiton}
\newtheorem{defi}[thm]{\bfseries Definition}
\theoremstyle{remark}
\newtheorem{example}[thm]{\bfseries Example}

\newtheorem{nota}[thm]{\bfseries Notation}

\newtheorem{rem}[thm]{\bfseries Remark}

\newcommand{\iso}{\cong}
\DeclareMathSymbol{\Z}{\mathalpha}{AMSb}{"5A} 
\DeclareMathSymbol{\PP}{\mathalpha}{AMSb}{"50} 
\DeclareMathSymbol{\Q}{\mathalpha}{AMSb}{"51}
\DeclareMathSymbol{\N}{\mathalpha}{AMSb}{"4E}
\DeclareMathSymbol{\R}{\mathalpha}{AMSb}{"52}

\newcommand{\D}{\mathbb{D}}
\DeclareMathOperator{\Min}{Min}
\newcommand{\ZZ}{\ensuremath{\mathbb{Z}}}
\newcommand{\RR}{\ensuremath{\mathbb{R}}}

\renewcommand{\le}{\leqslant}
\renewcommand{\ge}{\geqslant}

\newcommand{\la}{\lambda}
\newcommand{\La}{\Lambda}

\newcommand{\ga}{\gamma}

\renewcommand{\leq}{\leqslant}
\renewcommand{\geq}{\geqslant}

\definecolor{mongris}{gray}{0.9}

\renewcommand{\bf}{\bfseries}
\renewcommand{\it}{\itshape}

\newcommand{\U}[1]{#1 ^{\times}} 

\newcommand{\Vor}{{\mathrm{Vor}}}

\def \Aut{{\rm Aut}}
\newcommand{\GL}{\mathit{GL}}
\newcommand{\SL}{\mathit{SL}}

\definecolor{mongris}{rgb}{0.9, 0.9, .9}

\usepackage[all]{xy}
\renewcommand{\geq}{\geqslant}
\renewcommand{\leq}{\leqslant}

\newcommand{\cS}{\mathcal{S}}

\newcommand{\et}{\text{\'et}}
\newcommand{\Spec}{\text{Spec}}
\newcommand{\Gal}{\text{Gal}}

\newcommand{\abs}[1]{\lvert#1\rvert} 
 
\newcommand{\ctl}{\centerline} 
\newcommand{\ifff}{{if and only if }} 
\newcommand{\nd}{{\text{ and }}} 
\newcommand{\noi}{{\noindent}} 
\DeclareMathOperator{\sgn}{sgn} 
\DeclareMathOperator{\cl}{cl} 
 
\newcommand{\lan}{{\langle}} \newcommand{\ran}{{\rangle}}

\newcommand{\stx}{\begin{smallmatrix}} \newcommand{\estx}{\end{smallmatrix}} 

\newcommand{\wdt}{\widetilde} 
\newcommand{\tld}{{\hbox{\!\raise-.4ex\hbox{\,$\,\wdt{}\,$\,}}}} 
 
\newcommand{\cB}{{\mathcal B}} 
\newcommand{\cC}{{\mathcal C}} 
 
\newcommand{\cL}{{\mathcal L}} 
\DeclareMathOperator{\End}{End} 
\DeclareMathOperator{\Mat}{Mat}

\newcommand{\SC}{\mathcal{S}} 
\newcommand{\St}{\mathrm{St}} 

\AtBeginDocument{%
	\def\MR#1{}
}

\begin{document}
\title[Voronoi complexes in higher dimensions]{Voronoi complexes in higher dimensions, cohomology of $GL_N(\Z)$ for $N\geq 8$ and the triviality of $K_8(\Z)$} 
\subjclass[2010]{11H55, 11F75, 11F06, 11Y99, 55N91, 19D50, 20J06}
\keywords{Perfect forms,  Voronoi complex, group cohomology, modular groups, Steinberg modules, K-theory of integers,
 well-rounded lattices}

\author{Mathieu Dutour Sikiri\'c}
\address{Rudjer Bo\v skovi\'c Institute, Bijeni\v cka 54, 10000 Zagreb, Croatia}
\email{mathieu.dutour@gmail.com}

\author{Philippe Elbaz-Vincent}
\address{Univ. Grenoble Alpes, CNRS, Institut Fourier, F-38000 Grenoble, France}
\email{Philippe.Elbaz-Vincent@math.cnrs.fr}

\author{Alexander Kupers}
\address{Harvard University, Department of Mathematics,
	One Oxford Street,	Cambridge, 02138 MA, USA}
\email{kupers@math.harvard.edu}

\author{Jacques Martinet}
\address{Universit\'e de Bordeaux,
  Institut de Math\'ematiques,
  351, cours de la Lib\'e\-ration,
  33405 Talence cedex, France}
\email{Jacques.Martinet@math.cnrs.fr}

\renewcommand{\shortauthors}{Philippe Elbaz-Vincent et.~al.}

\begin{abstract} 
	We enumerate the low dimensional cells in the Voronoi cell complexes
	attached to the modular groups $\SL_N(\Z)$ and $\GL_N(\Z)$ for $N=8,9,10,11$, using quotient sublattices techniques for $N=8,9$  and linear programming methods for higher dimensions. These enumerations allow us to compute some cohomology of these groups and prove that $K_8(\Z) = 0$.
\end{abstract}

\maketitle

 
 Let $N\geq 1$ be an integer and let $\SL_N(\Z)$ be the modular group of integral matrices
 with determinant one. Our goal is to compute its cohomology groups with trivial coefficients, i.e.~$H^q(\SL_N(\Z);\Z)$. These are known in the cases $N \leq 7$: $N=2$ is classical  (e.g. II.7, Ex.3 of ~\cite{B}), $N=3$ is due to Soul\'e \cite{Soule-SL3}, $N=4$ is due to Lee and Szczarba \cite{LS}, and $N=5,6,7$ are due to Elbaz-Vincent, Gangl and Soul\'e \cite{EGS}.

 In Theorem 
\ref{cohomology} below, we give partial information for the cases $N=8,9,10$. For these calculations we follow mainly the methods of \cite{LS} and \cite{EGS}, by investigating the Voronoi complexes associated to these modular groups.

Recall that a \emph{perfect form} in $N$ variables is a positive definite real quadratic form $h$ on $\R^N$ which is uniquely determined (up to a scalar) by its set of integral minimal vectors.
Voronoi proved in \cite{Vo} that there are finitely many perfect forms of rank $N$, modulo the action of $\SL_N(\Z)$.
These are known for $N\leq 8$ (see \S\ref{sec1} below). These finitely many orbits of perfect forms give the top-dimensional generators in the Voronoi complex; the rest of the complex is constructed from these perfect forms. Unfortunately, we cannot work with the full Voronoi complex  for $N=8$ due to its size, which 
is beyond our computing capabilities,
and we do not have complete information for $N>8$.
However, it turns out that it \emph{is} possible to obtain partial information on the top and bottom parts
of the Voronoi complexes for $8 \leq N \leq 10$ and conjecturally $N=11$.
In particular, we can enumerate the cells of lowest dimension explicitly
using methods based on sublattices and relative index (cf.~\S\ref{martinet}). For other cases, we can use linear
programming in order to full enumeration
in given cellular dimensions (cf.~\S\ref{dutour}). 

Voronoi used perfect forms to define a cell decomposition of the space $X_N^*$ of positive real quadratic forms, whose kernel is defined over $\Q$. This cell decomposition (cf. \S\ref{sec2}) is invariant under $\SL_N(\Z)$, hence it can be used to give a chain complex which computes the equivariant homology of $X_N^*$ modulo its boundary: this is the Voronoi complex. On the other hand, this equivariant homology
turns out to be isomorphic to the groups $H_q\big(\SL_N(\Z); \St_N\big)$, where $\St_N$ is the Steinberg module
(see \cite{BS} and \S\ref{ssec2.4} below).
Finally, Borel--Serre duality asserts that the homology $H_*\big(\SL_N(\Z); \St_N\big)$ is dual to 
the cohomology $H^*\big(\SL_N(\Z);\Z\big)$ (modulo 
torsion at primes $\leq N+1$). Thus the results mentioned above give partial information about the cohomology of modular groups.

We will use this to obtain information about the algebraic K-theory of the integers. In particular, we will prove that the group $K_8(\Z)$ is trivial (Theorem \ref{K8Z}) and discuss its consequences for the Kummer--Vandiver conjecture (cf.~\S\ref{KV}).

\smallskip {\bf Organization of paper:} 
In \S \ref{sec1}, we recall the Voronoi theory of perfect forms and the  Voronoi complex
 which computes the homology groups 
$H_q\big(\Gamma,\St_\Gamma\big)$ with $\Gamma=\SL_N(\Z)$ or $\GL_N(\Z)$.
In \S \ref{sec3}, we give an explicit enumeration of the low dimensional cells 
of the Voronoi complexes associated to $\Gamma$. In \S \ref{sec4}, we present another method based on linear programming.
In \S\ref{sec5} we give
a partial description of the Voronoi complex associated to modular groups of rank $N=8$ up to $12$.
In \S \ref{sec6} we compute some homology groups of $\Gamma$ with coefficients in the Steinberg module and 
we explain how to compute part of the cohomology of $\SL_N(\Z)$ and $GL_N(\Z)$
(modulo torsion) for $N\geq 8$. 
In \S\ref{sec7}, we use these results to get some information on $K_m(\Z)$ for $m\geq 8$ and in particular show that $K_8(\Z)$ is trivial.
In \S\ref{sec8} we give some arithmetic applications.

Some of the enumerations of configurations of vectors had already been announced and used in \cite{StableBetti}.
Results concerning the triviality of $K_8(\Z)$ had already been announced in \cite{pev16,kupersshort}.

\smallskip
{\bf Acknowledgments:} Philippe Elbaz-Vincent is  partially supported by the French National Research Agency in the framework 
of the Investissements 
d'Avenir program (ANR-15-IDEX-02). Furthermore part of this work was done during his stay at the Hausdorff Research
Institute for Mathematics during the trimester ``Periods in Number Theory, Algebraic Geometry and Physics''. Alexander Kupers was supported by the Danish National Research Foundation through the Centre for Symmetry and Deformation (DNRF92) and by the European Research Council (ERC) under the European Union's Horizon 2020 research and innovation programme (grant agreement No.\, 682922) and  is currently supported by NSF grant DMS-1803766.

\tableofcontents

\section{The Voronoi reduction theory}\label{sec1} In this Section we recall some aspects of the Voronoi reduction theory \cite{Vo,martinet}.

\subsection{Perfect forms}\label{ssec1.1} Let $N \geq 2$ be an integer. We let $C_N$ be the set of 
positive definite real quadratic forms in $N$ variables. Given $h \in C_N$, let $m(h)$ be the finite set of minimal vectors of $h$,
i.e. vectors $v \in \Z^N$, $v \ne 0$, such that $h(v)$ is minimal. A form $h$ is called {\it perfect} 
when $m(h)$ determines $h$ up to scalar: if $h' \in C_N$ is such that $m(h') = m(h)$, then $h'$ is proportional to $h$. 
\begin{example}
 The form $h(x,y)=x^2+y^2$ has minimum 1 and precisely 4~minimal vectors $\pm (1,0)$ and $\pm (0,1)$. 
 This form is not perfect, because there is
 an infinite number of positive definite quadratic forms having these minimal vectors, namely the forms $h(x,y)=x^2+a x y+y^2$
where $a$ is a non-negative real number less than 1.
By contrast, the form $h(x,y)=x^2+x y+y^2$ has also minimum 1 and has exactly 6 minimal vectors, viz.~the ones 
above and $\pm (1,-1)$. This form is perfect, the associated lattice is the ``honeycomb lattice''.
\end{example}

 Denote by $C_N^*$ the set of non-negative real quadratic forms on ${\mathbb R}^N$ the kernel of which
 is spanned by a proper linear subspace of ${\mathbb Q}^N$, by $X_N^*$ the quotient of $C_N^*$ by positive
 real homotheties, and by $\pi : C_N^* \to X_N^*$ the projection. Let $X_N = \pi (C_N)$ and $\partial X_N^* = X_N^* - X_N$.
 Let $\Gamma$ be either $GL_N (\Z)$ or $\SL_N (\Z)$. The group $\Gamma$ acts on $C_N^*$ and $X_N^*$ on the right by the formula
\[
h \cdot \gamma = \gamma^t \, h \, \gamma \, , \quad \gamma \in \Gamma \, , \ h \in C_N^* \, ,
\]
where $h$ is viewed as a symmetric matrix and $\gamma^t$ is the transpose of the matrix $\gamma$.
Voronoi proved that there are only finitely many perfect forms modulo the action of $\Gamma$ and
multiplication by positive real numbers (\cite{Vo}, Thm.~p.110).\\
Table \ref{TableNumberExtremeForms} gives the current state of the art on the enumeration of perfect forms.

\begin{table}[h]
  \begin{tabular}{cccccccccc}
    \toprule
    rank  & 1 & 2 & 3 & 4 & 5 & 6 & 7 & 8  & 9 \\
    \midrule
    \# classes & 1 & 1 & 1 & 2 & 3 & 7& 33 & 10916 & $\geqslant 2.3\times 10^7$ \\
    \bottomrule
  \end{tabular}  
	\caption{Known results on the number of perfect forms up to dimension $9$}
  \label{TableNumberExtremeForms}
\end{table}

The classification of perfect forms of rank 8 was achieved by Dutour Sikiri\'c, Sch\"urmann and \nobreak{Vallentin} \cite{dsv,Sch}.
Partial results for dimension $9$ are reported in \cite{woerdenMasterThesis}. The corresponding classification 
for rank 7 was completed by Jaquet \cite{jaquet}, for rank 6 by Barnes \cite{barnes}, 
for rank 5 and 4 by Korkine and Zolotarev \cite{zolotarev77,zolotarev72}, 
for dimension 3 by Gauss \cite{gauss} and for dimension 2 by Lagrange \cite{lagrange}. We refer to the book of Martinet \cite{martinet} for more details on the results up to rank 7.
While the classification of perfect forms of higher rank is not well understood, we know from Bacher \cite{Bacher17}
that the number of representatives grows at least  exponentially with the rank and 
from van Woerden\cite{woerden2019upper} is bounded by $e^{O(d^2\log(d))}$ for perfect forms of rank $d$.

\subsection{The Voronoi complex}\label{sec2}

\begin{nota} For any positive integer $n$ we let ${\mathcal S}_n$ be the class of finite abelian groups the order 
	of which has only prime factors less than or equal to $n$.\end{nota}
\subsubsection{The cell complex}\label{ssec2.1}	
Given $v \in \Z^N - \{ 0 \}$ we let $\hat v \in C_N^*$ be the form defined by
\[
\hat v (x) = (v \mid x)^2 \, , \ x \in {\mathbb R}^N \, ,
\]
where $(v \mid x)$ is the scalar product of $v$ and $x$. The {\it convex hull in }$X_N^*$ of a finite subset
$B \subset \Z^N - \{ \mathbf 0 \}$ is the subset of $X_N^*$ which is the image under $\pi$ of the quadratic 
forms $\underset{j}{\sum} \, \lambda_j \, \hat{v_j}\in C_N^*$, where $v_j \in B$ and $\lambda_j \geq 0$. 
For any perfect form $h$, we let $\sigma (h) \subset X_N^*$ be the convex hull of the set $m(h)$ of its minimal vectors. 
Vorono{i} proved in \cite{Vo}, \S\S8-15, that the cells $\sigma (h)$ and their intersections, as $h$ runs over all 
perfect forms, define a cell decomposition of $X_N^*$, which is invariant under the action of $\Gamma$. We endow $X_N^*$ with 
the corresponding $CW$-topology. If $\tau$ is a closed cell in $X_N^*$ and $h$ a perfect form with $\tau \subset \sigma (h)$, 
we let $m(\tau)$ be the set of vectors $v$ in $m(h)$ such that $\hat v$ lies in $\tau$. Any closed cell $\tau$ is the convex 
hull of $m(\tau)$, and for any two closed cells $\tau$, $\tau'$ in $X_N^*$ we have  $m(\tau) \cap m(\tau') = m (\tau \cap \tau')$.

We shall now recall an explicit description of the Voronoi complex and its the differential from \cite{EGS}.

Let $d(N) = N(N+1)/2-1$ be the dimension of $X_N^*$ and $n \leq d(N)$ a natural integer. We denote 
by $\Sigma_n^\star= \Sigma_n^\star(\Gamma)$ a set of representatives, 
modulo the action of $\Gamma$, of those cells of dimension $n$ in $X_N^*$ which meet $X_N$, 
and by $\Sigma_n =\Sigma_n(\Gamma) \subset \Sigma_n^\star(\Gamma)$ the cells $\sigma$ 
for which any element of the stabilizer $\Gamma_{\sigma}$ 
of $\sigma$ in $\Gamma$ preserves orientation. 

Let $V_n$ be the free abelian group generated by $\Sigma_n$.
We define as follows a map
\[
d_n \colon V_n \to V_{n-1} \, .
\]

For each closed cell $\sigma$ in $X_N^*$ we fix an orientation of $\sigma$, i.e. an orientation of the real vector 
space ${\R} (\sigma)$ of symmetric matrices spanned by the forms $\hat v$ with $v \in m(\sigma)$. Let $\sigma \in \Sigma_n$ 
and let $\tau'$ be a face of $\sigma$ which is equivalent under $\Gamma$ to an element in $\Sigma_{n-1}$ (i.e. $\tau'$ neither lies 
on the boundary nor has elements in its stabilizer reversing the orientation). Given a positive basis $B'$ of ${\mathbb R} (\tau')$ 
we get a basis $B$ of ${\mathbb R} (\sigma)\supset {\mathbb R} (\tau')$ by appending to $B'$ a vector $\hat v$, where 
$v \in m(\sigma) - m(\tau')$. We let $\varepsilon (\tau' , \sigma) = \pm 1$ be the sign of the orientation of $B$ in the 
oriented vector space ${\mathbb R} (\sigma)$ (this sign does not depend on the choice of $v$).

\smallskip

Next, let $\tau \in \Sigma_{n-1}$ be the (unique) cell equivalent to $\tau'$ and let $\gamma\in\Gamma$ be 
such that $\tau'= \tau \cdot \gamma$. We define $\eta (\tau , \tau') = 1$ (resp. $\eta (\tau , \tau') = -1$) 
when $\gamma$ is compatible (resp. incompatible) with the chosen orientations of ${\mathbb R} (\tau)$ and ${\mathbb R} (\tau')$.

\smallskip

Finally, if $\sigma \in \Sigma_n$ and $\tau \in \Sigma_{n-1}$, we define the incidence number $[\sigma:\tau]$ 
for the Voronoi complex as
\begin{equation}
 [\sigma:\tau]=\sum_{\tau'} \eta (\tau , \tau') \, \varepsilon (\tau' , \sigma)\, ,
\end{equation}
where $\tau'$ runs through the set of faces of $\sigma$ which are equivalent to $\tau$. If $\tau$ is
not equivalent to a face of $\sigma$, we set $[\sigma:\tau]=0$.
The following map is thus well defined
\begin{equation}
\label{eq1}
d_n (\sigma) = \sum_{\tau \in \Sigma_{n-1}} [\sigma:\tau] \, \tau \, .
\end{equation}
It turns out that the map $d$ generalizes the usual differential of regular CW-complex 
to the case of the Voronoi complex (which is not regular CW-complex).

\subsubsection{The associated equivariant spectral sequence}\label{ssec2.2} 

According to Section VII.7 of \cite{B}, 
there is a spectral sequence $E_{pq}^r$ converging to the equivariant homology groups 
$H_{p+q}^{\Gamma} (X_N^* , \partial X_N^* ; \Z)$ of the pair $(X_N^* , \partial X_N^*)$, with $E^1$-page given by 
\[
E_{pq}^1 = \bigoplus_{\sigma \in \Sigma_p^\star} H_q (\Gamma_{\sigma} ; \Z_{\sigma}) \, ,
\]
where $\Z_{\sigma}$ is the orientation module of the cell $\sigma$ and, as above,
 $\Sigma_p^\star$ is a set of representatives, modulo $\Gamma$, of the $p$-cells $\sigma$ in $X_N^*$ which meet $X_N$. 
Notice that the action of $\Gamma_{\sigma}$ on  $\Z_{\sigma}$ is given by $\eta$ described above.
 Since $\sigma$ meets $X_N$, its stabilizer $\Gamma_{\sigma}$ is finite and, by Lemma \ref{lemma1} in \S\ref{sec4} below,
 the order of $\Gamma_{\sigma}$ is divisible only by primes $p \leq N+1$. Therefore, when $q$ is positive,
 the group $H_q (\Gamma_{\sigma} ; \Z_{\sigma})$ lies in ${\mathcal S}_{N+1}$. 

When $\Gamma_{\sigma}$ happens to contain an element which changes the orientation of $\sigma$, 
the group $H_0 (\Gamma_{\sigma} ; \Z_{\sigma})$ is killed by $2$, otherwise  $H_0 (\Gamma_{\sigma} ; \Z_{\sigma})\cong \Z_{\sigma}$.
Therefore, modulo ${\mathcal S}_2$, we have
\[
E_{n\,0}^1 = \bigoplus_{\sigma \in \Sigma_n} \Z_{\sigma} \, ,
\]
and the choice of an orientation for each cell $\sigma$ gives an isomorphism between $E_{n\,0}^1$ and $V_n$.

\begin{prop}\cite[\S3.3, p.591-592]{EGS} The differential
\[
d_n^1 \colon E_{n\,0}^1 \to E_{n-1,0}^1
\]
coincides, 
up to sign, with the map $d_n$ defined in \ref{ssec2.1}. 
\end{prop}

As pointed out on p.589 of \cite{EGS}, the identity $d_{n-1} \circ d_n = 0$  gives us a non-trivial test of our explicit computations.


\begin{nota} The resulting complex $(V_\bullet,d_\bullet)$ is denoted by $\Vor_\Gamma$, and is called  the {\em Voronoi complex}.\end{nota}

\subsubsection{The Steinberg module}\label{ssec2.4} 

Let $T_N$ be the spherical Tits building of $\SL_N$ over ${\mathbb Q}$, i.e.~the simplicial set obtained as the nerve of the ordered set of non-zero proper linear subspaces of ${\mathbb Q}^N$. The Solomon--Tits theorem says that $T_N$ is homotopy equivalent to wedge of $(N-2)$-spheres, see Theorem IV.5.2 of \cite{B}. Thus the reduced homology $\tilde H_q (T_N,\Z)$ of $T_N$ with integral coefficients is zero except when $q = N-2$, in which case
\[
\tilde H_{N-2} (T_N,\Z) \eqqcolon {\St}_N
\]
is by definition the \emph{Steinberg module}. According to Proposition 1 of \cite{SouleK4}, the relative homology 
groups $H_q (X_N^* , \partial X_N^* ; \Z)$ are zero except when $q = N-1$, and
\[
H_{N-1} (X_N^* , \partial X_N^* ; \Z) = {\St}_N \, .
\]
From this it follows that, for all $m \in {\N}$,
\begin{equation}
\label{fact1}
H_m^{\Gamma} (X_N^* , \partial X_N^* ; \Z) \iso H_{m-N+1} (\Gamma ; {\St}_N)
\end{equation}
(see e.g.~\S3.1 of \cite{SouleK4}). Combining this equality with the previous sections, we obtain:

\begin{prop}\label{vorsteinberg}
For arbitrary positive integers $N>1$ and $m$. We have the following isomorphism modulo ${\cS}_{N+1}$
\begin{equation}
\label{eq3}
H_{m-N+1} (\Gamma ; {\St}_N) \iso  H_m (\Vor_\Gamma) \, \mod  {\cS}_{N+1}.
\end{equation}
\end{prop}

\section{Small cells of quadratic forms}\label{sec3}\label{martinet}
The method described in \cite{M1}, 
Sections 9.2 and~9.3, though not very efficient, 
can be used to classify cells of dimension $\frac{n(n+1)}2-t$ 
for small values of $t$; 
these are the ``small cells'' referred to in the title.

\subsection{Minimal classes and the perfection rank.} 

Our aim is the study of the Voronoi complex  in a given 
(cellular) dimension~$n$. We shall make use of Watson's index theory, a theory
which is better understood in terms of lattices. 
For this reason we first recall some data of the ``dictionary'' 
which link lattices with quadratic forms.

\smallskip 

We identify a quadratic form $q$ with the $n\times n$ symmetric 
matrix $A$ such that $q(x)=X^{t} A X$, where $X$ is the column-vector 
of the components of $x\in\R^n$. Let $E$ be an $n$-dimensional Euclidean 
space, on which the {\em norm of $x$ is $N(x)=x\cdot x$}. 
With a lattice $\La\subset E$ (discrete subgroup of $E$ 
of rank~$n$) and a basis $\cB=(e_1,\dots,e_n)$ for $\La$ over $\Z$, 
we associate the Gram matrix $A=(e_i\cdot e_j)$ of $\cB$ 
and the corresponding positive, definite quadratic form $q$. 
The {\em determinant of $\La$} is $\det(A)$, 
the {\em discriminant} of~$q$. 
The {\em minimum} of $\La$, its set $m(\La)$ 
of {\em minimal vectors} correspond to the same notions 
for quadratic forms. 
We denote by $s$ the number of {\em pairs of minimal vectors} 
of~$\La$ (or of~$q$). Besides this ``kissing number'' 
an important invariant is the {\em perfection rank~$r$}, 
the definition of which we recall now. 
Given a line $D\subset E$, we denote by $p_D\in\End^s(E)$ 
the orthogonal projection to~$D$, and write $p_x$ if $D=\R\,x$ 
for some $x\ne 0$ in~$E$. Note the formulae 

\smallskip
\noi\ctl{$p_x(y)=\frac{x\cdot y}{x\cdot x}\,x\ \nd\ 
\Mat(\cB^*,\cB,p_x)=X X^{t}$} 

\smallskip\noi 
where $\cB^*$ is the dual basis to $\cB$, 
defined by $e_i\cdot e_j^*=\delta_{i,j}$. 

\begin{defi} \label{defpef} {\rm 
The {\em perfection rank of a family $D_1,\dots,D_s$ of lines in $E$} 
is the dimension of the span in $End^s(E)$ 
of the projections~$p_{D_i}$. 
A {\em perfection relation on the set $\{D_i\}$} is a non-trivial 
$\R$-linear relation of the form $\sum\la_i\,p_{D_i}=0$. 
The {\em perfection rank~$r$ of a lattice $\La$} 
is the perfection rank of the set of lines $\R\,x$, $x\in S(\La)$; 
that of a {\em quadratic form $q$} is the rank in $\mathrm{Sym}_n(\R)$ 
of the set $\{X X^{t}\}$, $X\in m(q)$.} 
\end{defi} 

We partition the space $\cL$ of lattices into {\em minimal classes} 
by the relation $\La\sim\La'$ if and only if there exists $u\in\GL(E)$ such that $u(\La)=\La'$ and $u(m(\La))=m(\La')$, and order the minimal classes by the relation 
$\cC\prec\cC'$ if and only if there exists $\La\in\cC$ and $\La'\in\cC'$ such that  $m(\La)\subset m(\La')$. 

The dictionary above establishes a one-to-one correspondence 
between the set of minimal classes on the one hand, 
and the set of cells up to equivalence of (positive, definite) 
quadratic forms having a given minimum on the other hand. 
These are finite sets. 

\smallskip 

In the sequel, we restrict ourselves to {\em well-rounded lattices} 
(or {\em forms}), those for which the minimal vectors span $E$. 
The following proposition provides an easy test to decide whether 
two lattices belong to the same minimal class. 
Given a lattice $\La$ and a basis $\cB$ for $\La$, 
and a set $S$' of representatives of pairs of minimal vectors 
of $S\La$, let $T$ be the $n\times s$ matrix of the components 
of the vectors of $S'$ on $\cB$, and let $B=T\,T^{t}$. 
This is the {\em barycenter matrix of $(\La,\cB)$}. 
Its equivalence class under $\GL_n(\Z)$ does not depend on $\cB$. 

\begin{prop} \label{propbarycentre} 
Two lattices belong to the same minimal class \ifff 
their barycenter matrices define the same class under $\GL_n(\Z)$. 
\end{prop} 

\begin{proof} 
This is Proposition~9.7.2 of \cite{martinet}.
\end{proof} 

\begin{lem} \label{lem2n} 
Any perfection relation in $\End^s(E)$ between projections 
to vectors of $E$ may be written in the form 

\ctl{$\sum _{x\in S}\,\la_x p_x=\sum _{y\in T}\,\mu_y p_y$} 

\smallskip\noi 
where $\la_x,\,\mu_y$ are strictly positive and $S$ and $T$ span 
the same subspace of~$E$. 
\end{lem}

\begin{proof} 
Getting rid of the zero coefficients, we obtain a relation 
of this kind for convenient subsets $S,T$ of~$E$, 
and we may moreover assume that the vectors $x,y$ have norm~$1$. 
Applying this relation on a vector $z\in T^\perp$ 
and taking the scalar products with $z$, we obtain the equality 

\ctl{$\sum_{x\in S}\,\la_x\,(x\cdot z)^2=0$} 

\smallskip\noi 
which shows that $z$ also belongs to $S^\perp$. We have thus proved 
the inclusion 
$T^\perp\subset S^\perp$, i.e., $\lan S\ran \subset\lan T\ran$, 
and exchanging $S$ and $T$ shows that $\lan S\ran=\lan T\ran$. 
\end{proof} 

\begin{rem} \label{remorth} 
By the lemma above, a perfection relation with non-zero coefficients 
between vectors which span an $m$-dimensional subset $F$ of $E$ 
involves at least $2m$ vectors. This is optimal for all $m\ge 2$, 
as shown by the union of two orthogonal bases $\cB$ and $\cB'$ 
for~$F$, since the sum of the orthogonal projections to the vectors 
of $\cB$ and $\cB'$ both add to the orthogonal projection to~$F$.
This construction of perfection relations accounts for those 
of the lattice $\D_4$, since $S(\D_4)$ is the union of three 
orthogonal frames.
\end{rem} 

\begin{thm} \label{th(n+s)} 
The perfection rank $r$ of a well-rounded $n$-dimensional lattice 
with kissing number $s\le n+5$ is equal to~$s$. 
\end{thm}

\begin{proof} 
By definition of the perfection rank we have $r\le s$, 
and if $r<s$, there exists a perfection relation with support 
$2m$ vectors of some $m$-dimensional subspace $F$ of $E$. 
We have $s\ge n+m$, hence $m\le 5$. Now for a lattice $L$ 
of dimension $n\le 5$, one has $s=n$ except if $L\sim\D_4$ 
or if $n=5$ and $L$ has a $\D_4$-section having the same minimum. 
Since $s(\D_4)=12$, we then have $s\ge n+8$, a contradiction. 
\end{proof} 
Notice that in the above theorem the bound $s-n\le 5$ is optimal; see Example~\ref{exad=3} below.

\subsection{Watson's index theory and very small cells.} 

\subsubsection{Codes associated with well-rounded lattices.} 
Let $\La$ be a well-rounded lattice, let $e_1,\dots,e_n$ 
be $n$~independent minimal vectors of~$\La$, 
and let $\La'$ be the sublattice of $\La$ generated by the~$e_i$. 
Then the index $[\La:\La']$ is bounded from above (by $\ga_n^{n/2}$) 
and so is the annihilator~$d$ of $\La/\La'$. 
The {\em maximal index $\imath(\La)$} of a well-rounded lattice 
$\La$ is the largest possible value of $[\La:\La']$ for 
a pair $(\La,\La')$ as above. 

Every element of $\La$ can be written in the form 

\smallskip\ctl{ 
$x=\dfrac{a_1 e_1+\dots+a_n e_n}d,\ a_i\in\Z$\,,} 

\smallskip\noi 
and if $d>1$ the systems $(a_1,\dots,a_n)\!\!\mod d$ can be viewed 
as the words of a $\Z/d\Z$-code. The codes arising this way have been 
classified for $n$ up to~$8$ in \cite{M1} (which relies on previous work 
by Watson, Ryshkov and Zahareva) and for $n=9$ in \cite{K-M-S}. 
The paper \cite{M1} relied on calculations which where feasible 
essentially by hand (together with some checks made using \textsf{PARI} \cite{PARI2}).
This is no longer possible beyond dimension~$8$, and indeed \cite{K-M-S} needed
the use of a linear programming package, which was implemented on \textsf{MAGMA} \cite{MAGMA}.

\smallskip 

By averaging on the automorphism group of the code (see Proposition 8.5 of \cite{M1}), one proves that if some code $\cC$ of length~$n$ can be lifted to a pair $(\La,\La')$ (we then say that $\cC$ is {\em admissible}), then there exist $\La$ and $\La'$ which are invariant under $\Aut(\cC)$. Then the minimum $s_{\min}$ of $s$ is attained on such a lattice $\La$, and the minimal class of $\La$ depends uniquely on~$\cC$. 

By inspection of the tables of \cite{M1}(Tableau 11.1) and \cite{K-M-S}(Tables 2-10), one proves: 

\begin{prop} \label{prop(n+6)} 
Let $d\ge 2$ and $n\le 9$, and let $\cC$ be an admissible 
$\Z/d\Z$-code. Then either $\cC$ can be lifted to a pair 
$(\La,\La')$ with $m(\La)=\{\pm e_i\}$, or we have 
$s(\La)\ge n+6$ for every lift of~$\cC$. \qed 
\end{prop}

Before going further we give some more precise results 
on the index theory. The most useful tool 
is {\em Watson's identity}, relying vectors $e_1,\dots,e_n$ 
of a basis for $E$, a vector $e=\frac{a_1 e_1+\dots+a_n e_n}d$, 
and the vectors $e'_i:=e-\sgn(a_i)e_i\,$: 
\[\left(\Big(\sum_{i=1}^n|a_i|\Big)-2d\right)N(e)=
\sum_{i=1}^n|a_i|\big(N(e'_i)-N(e_i)\big)\,.\]

Applied to minimal vectors $e_i$ of a lattice $\La=\lan e_i,e\ran$, 
this proves the lower bound $\sum\,\abs{a_i}\ge 2d$, 
and moreover shows that the vectors $e'_i$ for which $a_i\ne 0$ 
are minimal whenever $\sum\,\abs{a_i}=2d$. 

\begin{example} \label{exad=3} 
Let $n=6$, $d=3$ and let $(e_1,\dots,e_6)$ be a basis 
for $E$ with $N(e_i)=1$ and constant scalar products 
$e_i\cdot e_j=t$. Let 
$\La=\lan e_1,\dots,e_6,e\ran$ where $e=\frac{e_1+\dots+e_6}3$. 
Then for $\frac 1{10}<t<\frac 14$ (e.g., $t=\frac 15$), 
we have $\min\La=1$, $m(\La)=\{\pm e_i,\pm e'_i\}$, and the perfection 
relations are proportional to $\sum p_{e_i}=\sum p_{e'_i}$, 
so that $s=12=n+6$ and $r=11=s-1$. This is a consequence of the fact that we 
 can associate  
canonically a perfection relation with every 
Watson identity \cite{BergeMartinet}(Proposition 2.5).
\end{example} 

\subsubsection{Primitive minimal classes (or cells).} 

Let $\cC$ be a minimal class. For every lattice $\La\in\cC$, 
$\wdt\La=\La\perp m\Z$, where $m=\min\La$, is a lattice in $E\times\R$, 
which defines a minimal class $\cC'=\cC+\Z$ of dimension $n+1$, 
containing all direct sums $\La\oplus m\Z$ close enough to $\wdt\La$. 
We say that $\cC$ is {\em primitive} if it does not extend 
by this process a class of dimension~$n-1$. 

\smallskip 

Among $n$-dimensional (well-rounded) minimal classes, that of $\Z^n$, 
which has $\imath=1$ and $s-n=0$, plays a special role. Indeed,  
let $\cC$ be a minimal class and let $e_1,\dots,e_n$ be independent
minimal vectors of some lattice $\La\in\cC$, and let $\La'\subset\La$ 
be the lattice with basis $(e_1,\dots,e_n)$. 
Let $I_1\subset\{1,\dots,n\}$ be the support of the code defined by 
$(\La,\La')$ ($I_1=\emptyset$ if $\La=\La'$) and let $I_2$ 
be the set of subscripts which occur as components of minimal vectors 
distinct from the $\pm e_i$ ($I_2=\emptyset$ if $s(\La)=n$). 
Set $I=I_1\cup I_2$ and $m=\abs{I}$. 
Then if $\cC\ne\cl(\Z^n)$, $I$ is not empty, so that $\cC$ 
extends a minimal class of dimension~$m\ge 2$. 

\smallskip 

To list minimal classes up to equivalence it suffices to list 
those which are primitive and then complete the list with those 
of the form $\cC'\oplus\Z$ for some class $\cC'$ of dimension $n-1$. 

\subsubsection{Classes with $s=n$.} 
\begin{thm} The numbers of minimal classes of well-rounded 
lattices with $s=n$ and $n\le 9$ having a given index $\imath$ 
are displayed in Table \ref{ResultQuotient}.
\end{thm}
\begin{proof} 

The lattices $\La$ with $s=n$  contain a unique sublattice $\La'$ 
generated by minimal vectors of $\La$, and $\La'$ has a unique basis
up to permutation and changes of signs. Hence the classification 
of minimal classes coincide with the index classification. 
We mainly need to consider the results from \cite{M1} and \cite{K-M-S}.
According to the previous section, for index 2 (resp. 3) there
is one primitive class if $n \ge 5$ (resp. $n \ge  7$), hence $n-4$ (resp. $n-6$)
cells. Similarly for cyclic quotients of order 4 and $n \ge 9$ there are $n-4$
primitive cells, but only three if $n = 8$, one if $n = 7$ and none if $n \le 6$,
which for $n = 7, 8, 9$ yields 1, 4, and 9 cells of cyclic type with $\imath = 4$.
In the case $\imath=4$ but with $\La/\La'$ of type $2^2$ (i.e. $\imath=2^2$), the condition is that 
 the corresponding binary code must be of weight $w \ge  5$. This implies $\ell \ge 8$, and if $\ell = 8$ (resp.
$\ell = 9$), we are left with one code, with weight distribution $(5^2 , 6)$ (resp.
three codes, with weight distributions $(5^2 , 8)$, $(5, 6, 7)$ and $(6^3 )$). 
As a result we got four primitive cells.
Finally we just have to read in the tables of \cite{M1}(Tableau 11.1) and \cite{K-M-S}(Tables 2-10) 
for index $\imath \ge 5$  and $n\le 9$ for all  admissible codes for which $s=n$ is possible.
\end{proof} 
{\small 
\begin{table}[h]
\begin{tabular}{cccccccccccc}
\toprule
$n\backslash \imath$&  $1$  &  $2$  &  $3$  &   $4$  &  $2^2$  &  $5$  &  $6$  &   
$7$  &  $8$  &  $4\cdot 2$  & \# minimal classes \\ 
\midrule
$\le 4$ &  $1$  &  $0$  &   $0$  &  $0$  &  $0$  &  $0$  &   $0$  &
$0$  &  $0$  &  $0$ & $1$ \\ 
$5$   &  $1$  &  $1$  &   $0$  &  $0$  &  $0$  &  $0$  &   $0$  & 
$0$  &  $0$  &  $0$  & $2$ \\ 
$6$   &  $1$  &  $2$  &   $0$  &  $0$  &  $0$  &  $0$  &   $0$  & 
$0$  &  $0$  &  $0$  & $3$ \\ 
$7$   &  $1$  &  $3$  &   $1$  &  $1$  &  $0$  &  $0$  &   $0$  & 
$0$  &  $0$  &  $0$  & $6$ \\ 
$8$   &  $1$  &  $4$  &   $2$  &  $4$  &  $1$  &  $1$  &   $0$  & 
$0$  &  $0$  &  $0$  & $13$ \\ 
$9$   &  $1$  &  $5$  &   $3$  &  $9$  &  $4$  &  $4$  &   $9$  & 
$3$  &  $3$  &  $3$  & $44$ \\ 
\bottomrule
\end{tabular}
\quad\\[3pt]
\caption{Number of minimal classes of well-rounded lattices with $s=n$ and $n\leq 9$ according to the quotient.}\label{ResultQuotient}
\end{table}} 

To deal with slightly larger values of $s$, we first establish 
an as short as possible list of {\em a priori} possible index systems, 
then test for minimal equivalence the putative classes obtained 
from this list, and finally construct explicitly lattices 
in the putative class or prove that such a class does not exist. 
This third step is the more difficult and it is hopeless to deal with  
relatively large values of $s$ or of $n$ without using efficient 
linear programming methods, as seen in \S\ref{dutour}.


\subsection{Lattices with $s=1$ or $s=2$.} 

Consider a lattice $\La_0$ having a basis of minimal vectors 
$\cB=(e_1,\dots,e_n)$ and $t=s-n$ other minimal vectors 
$x_i=\sum_{i=1}^n\,a_{i,j} e_j$, {\small $1\le j\le t$}. 
The $m\times m$ determinants ($m\le\min(t,n)$) 
extracted from the matrix $(a_{i,j})$ are the 
{\em characteristic determinants} of Korkine and Zolotarev (cf. Chapter 6 from \cite{martinet}). 
With every characteristic determinant $d\ne 0$ they associate 
a lattice $\La'_0$ of index $\abs{d}$ in~$\La_0$. 
This shows that if $\imath(\La_0)=1$, then $d=0$ or $\pm 1$; 
in particular all $a_{i,j}$ are $0$ or $\pm 1$.

\section{Enumeration of configurations of vectors}\label{sec4}\label{dutour}

In this section we explain the algorithms used for enumerating the configurations of vectors
in dimension $n\in \{8, \dots, 12\}$ and rank $r=n$, $n+1$ or above.

Our approach computes first the configurations of vectors in rank $r=n$ and then from this enumeration
gets the configurations in rank $r=n+1$, then $n+2$ and so on.
The number of cases explodes with the dimension $n$ and rank $r$ as one expects and the computation is
thus quite slow. However, for the case $r=n$ an additional problem occurs: the known bounds  on the
determinant of vector configurations are suboptimal.
All  computations rely on the ability to test if a configuration of shortest vectors of a positive definite matrix can be derived from a given configuration of vectors.

\subsection{Testing realizability of vector family}

In \cite{K-M-S} an algorithm for testing realizability of vector families by solving linear programs was introduced. We describe below the needed improvements of our strategy in order to reach higher dimensions. We refer to \cite{schrijver} for an account of the classical theory of linear programming.

Given $m$ affine functions $(\phi_i)_{1\leq i\leq m}$ on $\RR^n$ and another function $\phi$
the {\em linear program} is to minimize $\phi(x)$ for $x$ subject to the constraints $\phi_i(x)\geq 0$.
Define ${\mathcal P} = \{x\in \RR^n, \mbox{~s.t.}\;\phi_i(x)\geq 0\}$. 
The linear program is called {\em feasible} if ${\mathcal P}\not= \emptyset$ and  the elements of ${\mathcal P}$
are called \emph{feasible solutions}.
If $\phi$ is bounded from below on ${\mathcal P}$ then the inferior limit  is denoted $opt({\mathcal P}, \phi)$
and is attained by one feasible solution. Any feasible solution will satisfy
$\phi(x) \geq opt({\mathcal P}, \phi)$

The {\em dual problem} is to maximize the value of $m$ such that there exist $\beta_i$ with
\begin{equation*}
\phi = m + \sum_{i=1}^m \beta_i \phi_i \mbox{~with~} \beta_i\geq 0
\end{equation*}

Any feasible solution of the dual problem will gives us $\phi(x) \geq m$ and the maximum value of
such $m$ will be exactly $opt({\mathcal P}, \phi)$ by a theorem of von Neumann \cite{schrijver}(\S7.4).
In other words, any feasible solution $(m, \beta_i)$ of the dual problem will give us $opt({\mathcal P}, \phi)$.

Let $A$ be an $n \times n$ matrix with real coefficients and set $A[v]:=v^t A v$ for any $v \in \R^n$.  Given a configuration of vectors ${\mathcal V}$ the basic linear program to be considered is
\begin{equation*}
  \begin{array}{rl}
    \mbox{minimize} & \lambda\\
    \mbox{with}     & \lambda=A[v] \mbox{~for~} v\in {\mathcal V}\\
    & A[v] \geq 1 \mbox{~for~} v\in \Z^n - \{0\} - {\mathcal V}
  \end{array}
\end{equation*}

If the optimal value satisfies $\lambda_{opt} < 1$ then ${\mathcal V}$ is realizable, otherwise no.

The main issue is that the above linear program has an infinity of defining inequalities and so instead we consider the program restricted to a finite subset, i.e. the linear inequalities:
\begin{equation*}
A[v] \geq 1 \mbox{~for~} v\in {\mathcal S} \mbox{~with~}
\mbox{${\mathcal S}$~finite~and~${\mathcal S}\subset \Z^n - \{0\} - {\mathcal V}$}.
\end{equation*}

It can happen that the equalities $\lambda=A[v]$ for $v\in {\mathcal V}$ has no solution with $\lambda\neq 0$.
In that case ${\mathcal V}$ is not realizable.

It can also happen that the linear program is unbounded that is solutions with arbitrarily negative value of $\lambda$ are feasible.
In that case we append $2 {\mathcal V}$ to ${\mathcal S}$.

Thus if those restrictions are implemented then the linear program has an optimal rational solution $A_{opt}({\mathcal S})$ 
of optimal value $\lambda_{opt}({\mathcal S})$.

According to the solution of the linear program we can derive following conclusions:
\begin{enumerate}
\item If $\lambda_{opt}({\mathcal S}) \geq 1$ then we can conclude that the vector configuration is not realizable.
\item If $\lambda_{opt}({\mathcal S}) < 1$ and $A_{opt}({\mathcal S})$ is positive definite then we 
compute $\Min(A_{opt}({\mathcal S}))$.
  \begin{enumerate}
  \item If $\Min(A_{opt}({\mathcal S})) = {\mathcal V}$ then the configuration is realizable
  \item Otherwise we cannot conclude. But we can insert the vectors in the difference 
  $\Min(A_{opt}({\mathcal S}))  -  {\mathcal V}$ into ${\mathcal S}$ and iterate.
  \end{enumerate}
\item If $\lambda_{opt}({\mathcal S}) < 1$ and $A_{opt}({\mathcal S})$ is not of full rank then we can
compute some integer vector in the kernel of $A_{opt}({\mathcal S})$ and insert them into ${\mathcal S}$ and iterate.
\item If $\lambda_{opt}({\mathcal S}) < 1$ and $A_{opt}({\mathcal S})$ is of full rank but not positive 
semidefinite then we can compute an integer vector $v$ such that $A_{opt}({\mathcal S}) [v] < \lambda_{opt}({\mathcal S})$
and insert it into ${\mathcal S}$ and iterate.
\end{enumerate}
Thus we can iterate until we obtain either feasibility of the vector configuration of unfeasibility. In practice a naive implementation of this algorithm can be very slow and we need to apply a number 
of improvements in order to get reasonable running time:
\begin{enumerate}
\item The dimension of the program is $n(n+1)/2 - r$ and this is quite large.
We can use symmetries in order to get smaller program.
Namely we compute the group of integral linear transformation preserving ${\mathcal V}$ and impose
that the matrix $A$ also satisfies this invariance.

\item Even after symmetry reduction the linear programs have many inequalities and are hard to solve.
In our implementation we use {\tt cdd} \cite{cdd} based on exact arithmetic and provides solutions of
the linear program and its dual in exact rational arithmetic.
However, {\tt cdd} uses the simplex algorithm and is very slow in some cases. Thus the idea is to use 
floating point arithmetic and the {\tt glpk} program \cite{glpk} which has better algorithm and can
solve linear programs in double precision.
From the approximate solution we can guess in most cases a feasible rational solution of the
linear program and its dual. If both gives the same value, then we have resolved our linear program.
If this approach fails, then we fall back to the more expensive in time {\tt cdd}.
In all cases, we only accept a solution if it has a corresponding dual solution.

\item If the matrix $A_{opt}({\mathcal S})$ is of full rank but not positive definite then there exists 
an eigenvector $w\in \R^n$ of eigenvalue $\alpha < 0$. We then use the sequence of vectors
\begin{equation*}
w^i = \left( Near(i w_1), Near(i w_2), \dots, Near(i w_n) \right).
\end{equation*}
with $Near(x)$ being the nearest integer to a real number $x$.
As $i$ increases $w^i$ approaches the direction of the vector $w$.
Thus there is an index $i_0$ such that $w^{i_0}\not= 0$ and $A_{opt}({\mathcal S}) [w^{i_0}] < \lambda_{opt}({\mathcal S})$
and this vector $w^{i_0}$ can be inserted into ${\mathcal S}$.
The problem is that in many cases the matrix $A_{opt}({\mathcal S})$ is very near to being positive definite
and the negative eigenvalue will be very small.
Thus we first try double precision with the {\tt Eigen} template matrix library \cite{eigenweb} for the computation
of the eigenvector and vector $w^{i_0}$.
If this fails then we use the arbitrary precision floating point library {\tt mpfr} \cite{mpfr},
still with {\tt Eigen} and progressively increase the number of digits until a solution is found.

\item An issue is for the initial vector set. In our implementation we set
\begin{equation*}
{\mathcal S} = {\mathcal V} \cup \left\lbrace x \pm e_i \pm e_j\mbox{~for~}x\in {\mathcal V} \mbox{~and~}1\leq i,j\leq n\right\rbrace - \{0\}
\end{equation*}
but there could be better choice for initial vector set.

\item We apply {\tt LLL} reduction \cite{cohenbook} to the vector configuration. Namely, we define a positive definite matrix
\begin{equation*}
A_{{\mathcal V}} = \sum_{v\in {\mathcal V}} v^t v
\end{equation*}
and apply the {\tt LLL} reduction to it in order to get a reduction matrix $P\in \GL_{n}(\ZZ)$.
The matrix $P\in \GL_n(\Z)$ is then used to reduce ${\mathcal V}$ as well.
The use of LLL reduction reduces the maximal size of the coefficients and dramatically reduces the number 
of iterations needed to get a result. Hence we use it systematically.
\end{enumerate}

When all those methods are implemented we manage to do the realizability tests in reasonable time.

A variant of the above mentioned realizability algorithm is to consider a family of vector ${\mathcal V}$
of rank $r$ and return a realizable configuration of vectors ${\mathcal W}$ of rank $r$ if it exists that
contains ${\mathcal V}$. It suffices in the case $\lambda_{opt}({\mathcal S}) < 1$ and $A_{opt}({\mathcal S})$
to distinguish between vectors of $\Min(A_{opt})$ that increases the rank and vectors that do not increase the
rank.

\subsection{Enumeration of configuration of vectors with $r=n$}
In \cite{K-M-S} it is proved that for a configuration $\{\pm v_1, \dots, \pm v_n\}$
of shortest vectors, we have
\begin{equation*}
\vert \det(v_1, \dots, v_n)\vert \le \sqrt{\gamma_n^n}
\end{equation*}
with $\gamma_n$ the Hermite constant in dimension $n$. As it turns out this upper
bound on the determinant is tight for dimension $n\leq 8$ but not in dimension $9$
and $10$. An additional problem is that $\gamma_n$ is known exactly only for $n\leq 8$
and $n=24$. Our strategy is thus to simply enumerate the vector configurations
up to the best upper bound that we have on the determinant.

For dimension $10$, combined with the known upper bound on $\gamma_{10}$ this
gives an upper bound of $59$ on the indexes of the relevant lattices \cite{K-M-S}.
If one uses the conjectured value of $\gamma_{10}$ then one gets $36$ as upper bound
in dimension $10$. It will turn out that the maximal possible determinant is $16$.

A key aspect of the enumeration is to enumerate first the cases where the quotient
$\ZZ^n / L$ has the structure of a prime cyclic group. This is important since if
a prime $p$ is unfeasible then any vector configurations with determinant divisible by $p$ 
is unfeasible as well. It is also important since our enumeration goes prime
by prime for composite determinants. For $d=p_1\times \dots \times p_r$ we
first do the enumeration of vector configurations of determinant $p_1$, then $p_1p_2$ and so on.

In the case of lattices of index $p$ with $p$ prime we consider a lattice $L$ spanned by
$e_1=(1, 0,\dots, 0)$, \dots, $e_n = (0, \dots, 0, 1)$ and
\begin{equation*}
e_{n+1} = \frac{1}{p} (a_1, \dots, a_n), a_i\in \ZZ
\end{equation*}
such that $(e_1, \dots, e_n)$ is the configuration of shortest vectors of a lattice.
By standard reductions, we can assume that
\begin{itemize}
\item $a_1\leq a_2 \leq \dots \leq a_n$
\item and $1\leq a_i\leq \left\lfloor p/2\right\rfloor$.
\end{itemize}
Since $p$ is prime $e_{n+1}$ can be replaced by $k e_{n+1}$ for any $1\leq k\leq p-1$.
Thus we can assume the vector $e_{n+1}$ to be lexicographically minimal among all possible
vectors.

It turns out that lexicographically minimal vector configurations can be enumerated by
exhaustive enumeration without having to store in memory the list of candidates.
The idea is as follows: if $(a_1, \dots, a_n)$ is lexicographically minimal, then
$(a_1, \dots, a_{n-1})$ is also lexicographically minimal. Lexicographic minimality
only requires $(p-1)n$ multiplications and reductions to be tested. Thus we enumerate
all configuration up to length $n-1$ and then extend this enumeration to length $n$ by
adding all possible feasible candidates.
For $n=10$ and $p=59$ we have $16301164$ possible vector configurations and for each
of them we test realizability.

When the enumeration for index $p$ is done we can continue the enumeration up to index
$p p'$ by taking all feasible lattices of index $p$ and considering all their sublattices
of index $p'$ up to action of the symmetry group. Thus we get a set of \emph{a priori} feasible lattices
for which we can apply our realizability algorithm and get a list of lattices of index
$p p'$.
For $n=10$ the most complex case of this kind is $49 = 7^2$.

By doing prime by prime up to $59$ in this way we are able to get all configurations
of shortest vectors in dimension $10$, we find $283$ different lattices.
For dimension $n=11$, we were only able to go up to index $45$ and we got in total
$6674$ possible sublattices. The list is not proved to be complete but it is reasonable
to conjecture that this list is complete since the maximum determinant of a realizable
vector configuration is $32$.
For dimension $n=12$, we managed only to go up to index $30$ and found $454576$
different vector configurations and it seems that this list is far from complete.

\subsection{Enumeration of configuration of vectors with $r>n$}

For the case $r=n+1$ and $r=n+2$ we have proved in \cite{smoothness}
that the relevant cones are simplicial. In \cite{smoothness} an algorithm
is given for getting the full list of configuration vectors in those ranks.
The only improvement to this algorithm is that the integer points are
obtained by an exhaustive enumeration procedure since {\tt zsolve} proved
too slow.

For rank $r=n+3$ and above simpliciality is not \emph{a priori} true though
it is expected to hold in rank $n+3$ and $n+4$. Thus we need a different
approach to the enumeration.
If we have a configuration of vectors ${\mathcal V}'$
in dimension $n$ and rank $r>n+2$ then it necessarily contains a $n$-dimensional
configurations ${\mathcal V}$ of rank $r-2$ and two $n$-dimensional
configurations ${\mathcal W}_1$ and ${\mathcal W}_2$ of rank $r-1$ such that
\begin{equation*}
{\mathcal V} \subset {\mathcal W}_i \subset {\mathcal V}'.
\end{equation*}

\noindent Our approach is as follows:
\begin{enumerate}
\item We first enumerate the configurations of vectors in dimension $n$ and
  rank $r-2$ and $r-1$.
\item We determine all the orbits of pairs $({\mathcal V}, {\mathcal W})$ with
  ${\mathcal V}$ of dimension $n$ rank $r-2$, ${\mathcal W}$ of dimension $n$ rank $r-1$.
\item For any configuration ${\mathcal V}$ of dimension $n$ rank $r-2$ there is a finite number
  of configurations of vectors of dimension $n$ rank $r-1$ containing it. We can thus
  enumerate all the pairs $({\mathcal W}_1, {\mathcal W}_2)$ containing ${\mathcal V}$ and check
  if ${\mathcal W}_1\cup {\mathcal W}_2$ is contained in a realizable family of vectors.
\end{enumerate}

\subsection{Obtained enumeration results}

By combining all above enumerations methods, we can obtain a number of orbits of perfect domains for small $r$ and $n$.
\begin{prop}\label{OrbitCones_Perfect}
The number of orbits of cones in the perfect cone decomposition for rank $r\leq 12$ and dimension $n$ at
most $11$ (the result for $r=n=11$ is conjectural) are given in Table \ref{NrOrbitCones_Perfect}.
\end{prop}
\enlargethispage*{1cm}
{\small
\begin{table}[h]
  \begin{tabular}{cccccccccc}
	\toprule
	$r\backslash n$
	& 4 & 5 & 6 & 7  & 8  & 9   & 10   &11    &12\\
	\midrule
	4         & 1 & 3 & 4 & 4  & 2  & 2   & 2    &-     &-\\
	
	5         &   & 2 & 5 & 10 & 16 & 23  & 25   &23    & 16\\
	
	6         &   &   & 3 & 10 & 28 & 71  & 162  &329   & 589\\
	
	7         &   &   &   & 6  & 28 & 115 & 467  &1882  & 7375\\
	
	8         &   &   &   &    & 13 & 106 & 783  &6167  & 50645\\
	
	9         &   &   &   &    &    & 44  & 759  &13437 &?\\
	
	10        &   &   &   &    &    &     & 283  &16062 &?\\
	
	11        &   &   &   &    &    &     &      &6674? &? \\
	\bottomrule
\end{tabular}
\quad\\[3pt]
\caption{Known number of orbits of cones in the perfect cone decomposition for rank $r\leq 12$ and dimension $n$ at
most $11$ (\emph{the result for $r=n=11$ is conjectural}).}
\label{NrOrbitCones_Perfect}
\end{table}
}

\begin{rem}
For dimension $n=11$, $12$ we do not have a full enumeration, however the partial configuration obtained are already
instructive. In all dimensions $n\leq 10$ the orientation of the well-rounded families of vectors with $s=n$ (i.e., the orientation of the associated cells of dimension $n$ in $X^*_n$) were found not to be preserved by their stabilizer.
However, this changes with $5$ well-rounded configurations known in dimension $11$ and $12$ in dimension $12$.
One such configuration in dimension $11$ is $e_{5}+e_{6}+e_{7}+e_{10}-e_{3}$, $e_{11}-e_{2}-e_{10}$, $e_{3}+e_{9}$,
$e_{4}+e_{6}+e_{8}-e_{2}-e_{9}$, $e_{1}+e_{2}+e_{3}$, $e_{4}+e_{7}+e_{11}-e_{6}$, $e_{8}+e_{9}-e_{7}$,
$e_{1}+e_{5}-e_{11}$, $e_{4}+e_{5}$, $e_{1}-e_{8}$, $e_{10}$ and it has a stabilizer of order $4$ which is the minimum
known so far. It seems reasonable to expect that there are well-rounded vector configurations with stabilizer of order $2$, i.e.~only antipodal operation.
We know just one well-rounded configuration in dimension $12$ whose orbit under $\GL_{12}(\ZZ)$ splits into two orbits
under $\SL_{12}(\ZZ)$.
One representative is $e_{6}-e_{1}-e_{2}-e_{3}-e_{4}$, $e_{6}-e_{7}-e_{8}$, $e_{9}-e_{3}-e_{6}-e_{10}-e_{11}-e_{12}$,
$e_{7}+e_{12}-e_{1}-e_{2}-e_{5}-e_{8}$, $e_{11}-e_{1}-e_{4}-e_{5}-e_{6}-e_{10}$, $e_{2}$, $e_{4}+e_{8}-e_{7}$, $e_{9}$,
$e_{10}$, $e_{3}+e_{5}+e_{7}+e_{9}-e_{1}$, $e_{11}$, $e_{12}$.
\end{rem}

\begin{rem} The above data for $r\leq 7$ recover the computations of \cite{EGS} (cf.~Figures 1 and 2) and \cite{LS}.
\end{rem}

\section{Homology of the Voronoi complexes}\label{sec5}
\subsection{Preliminaries}
Recall the following simple fact, cf.~p.602 of \cite{EGS}, which are relevant for understanding the action of $\GL_N(\Z)$ on $X_N$:
\begin{lem}\label{lemma1}\label{lemma2}
\begin{itemize}
\item Assume that $p$ is a prime and $g \in {\GL}_N ({\mathbb R})$ has order $p$. Then $p \leq N+1$.

\item The action of ${\GL}_N ({\mathbb R})$ on the symmetric space $X_N$ preserves its orientation if and only if $N$ is odd.

\end{itemize}
\end{lem}

\begin{rem}
We can give a more precise statement regarding the torsion in $\GL_N(\Z)$. Let $p$ an odd prime and $k$ a positive integer.
Set $\psi(p^k) = \varphi(p^k)$, $\psi(2^k) = \varphi(2^k)$ if $k>1$ and set $\psi(2)=\psi(1)=0$.
For an arbitrary $m = \prod_{\alpha} p_{\alpha}^{k_{\alpha}}$ define $\psi(m) = \sum_{\alpha} \psi(p_{\alpha}^{k_{\alpha}})$.
According to the crystallographic restriction theorem, an elementary proof of which is given as Theorem 2.7 of \cite{KPCR}, an element of order $m$ occurs in $\GL_{N}(\ZZ)$ if and only if $\psi(m)\leq N$. 
\end{rem}
\subsection{The Voronoi complexes in low dimensions} \label{homology_voronoi}
From the computations of \S\ref{sec4}, we deduce the following cardinalities.

\begin{prop} The number of low dimensional cells in the quotient $(X^*_N,\partial X^*_N)/\GL_N(\Z)$
is given by Table~\ref{tableBeyond8}.
\end{prop}

\begin{table}[h]
		\begin{tabular}{cccccc}
		\toprule
		$n$
		& 8 & 9 & 10 & 11  & 12 \\
		\midrule
		$\Sigma_n^\star\big(\GL_8(\Z)\big)$ & 13 & 106 & 783  &6167  & 50645\\
		$\Sigma_n\, \big(\GL_8(\Z)\big)$ & 0 & 0& 0& 0 & 0  \\
		$\Sigma_n^\star\big(\GL_9(\Z)\big)$ &   & 44  & 759  &13437 &?\\
		$\Sigma_n\,\big(\GL_9(\Z)\big)$ &  &  0&0 &0 & ?  \\
		$\Sigma_n^\star\big(\GL_{10}(\Z) \big)$ &  &   & 283  &16062 &?\\
		$\Sigma_n\,\big(\GL_{10}(\Z)\big)$ &  &  & 0 & 0 & ?\\
		\bottomrule
	\end{tabular}
	\caption{Cardinality of {$\Sigma_n$} and  {$\Sigma_n^\star$} for $N=8,9,10$ (empty slots denote zero).}
	\label{tableBeyond8}
\end{table}

The explicit data can be retrieved at the url \url{https://github.com/elbazvip/Voronoi-complexes-database}.
For convenience to the reader, we give below a set of representatives in the case $s=n=8$.

\begin{prop} A set of representatives for the 13 well-rounded cells of $\Sigma_8^\star\big(\GL_8(\Z)\big)$
is given by the following matrices:
{\renewcommand*{\arraystretch}{0.9}\setlength\arraycolsep{4pt}\fontsize{6.5}{6.5}\selectfont 
\begin{equation*}
\begin{gathered}
\begin{pmatrix}
0 &0 &0 &0 &0 &-1 &0 &0 \\ 
0 &-1 &0 &0 &0 &1 &0 &0 \\ 
0 &-1 &0 &1 &0 &1 &1 &0 \\ 
0 &0 &0 &1 &0 &1 &0 &0 \\ 
0 &-1 &-1 &1 &0 &0 &0 &0 \\ 
1 &1 &1 &0 &0 &0 &0 &0 \\ 
0 &2 &0 &-2 &0 &-1 &-1 &-1 \\ 
0 &-1 &-1 &1 &1 &1 &1 &1 
\end{pmatrix}
,
\begin{pmatrix}
0 &0 &1 &1 &1 &0 &0 &0 \\ 
0 &0 &-1 &-2 &-1 &0 &0 &0 \\ 
0 &0 &-2 &-2 &-2 &0 &-1 &0 \\ 
0 &0 &-1 &-1 &-1 &0 &0 &-1 \\ 
0 &-1 &-1 &-1 &-2 &0 &0 &0 \\ 
1 &1 &1 &1 &1 &0 &0 &0 \\ 
-1 &1 &3 &3 &3 &0 &1 &1 \\ 
0 &-1 &-2 &-2 &-2 &1 &0 &0 
\end{pmatrix}
,
\begin{pmatrix}
0 &1 &1 &0 &0 &0 &-1 &0 \\ 
0 &-1 &-2 &0 &0 &0 &1 &0 \\ 
0 &-1 &-2 &0 &0 &0 &1 &1 \\ 
0 &-1 &-1 &0 &0 &-1 &0 &0 \\ 
0 &-1 &-1 &0 &0 &1 &0 &0 \\ 
1 &1 &1 &0 &0 &0 &0 &0 \\ 
0 &1 &3 &0 &-1 &0 &0 &0 \\ 
-1 &-1 &-2 &1 &1 &0 &0 &0 
\end{pmatrix}
,\\
\begin{pmatrix}
0 &0 &1 &-1 &0 &0 &0 &-1 \\ 
0 &0 &-2 &1 &0 &0 &0 &1 \\ 
0 &-1 &-2 &1 &0 &1 &0 &1 \\ 
0 &-1 &-1 &0 &0 &0 &0 &1 \\ 
0 &-1 &-1 &1 &0 &0 &0 &0 \\ 
1 &1 &1 &0 &0 &0 &0 &0 \\ 
0 &2 &3 &-2 &0 &-1 &-1 &-2 \\ 
0 &-1 &-2 &1 &1 &1 &1 &1 
\end{pmatrix}
,
\begin{pmatrix}
-1 &-2 &0 &0 &1 &1 &1 &-1 \\ 
0 &1 &0 &0 &0 &-1 &-1 &0 \\ 
0 &1 &0 &0 &-1 &-1 &-1 &2 \\ 
0 &0 &1 &0 &0 &0 &0 &-1 \\ 
0 &-1 &0 &0 &0 &0 &1 &-1 \\ 
0 &1 &0 &0 &-1 &-1 &-2 &1 \\ 
0 &-1 &-1 &-1 &1 &1 &1 &0 \\ 
1 &1 &1 &1 &0 &0 &0 &0 
\end{pmatrix}
,
\begin{pmatrix}
0 &0 &0 &-1 &0 &0 &0 &1 \\ 
0 &-1 &0 &1 &0 &0 &1 &0 \\ 
0 &1 &0 &0 &-1 &0 &-1 &0 \\ 
1 &-1 &0 &1 &1 &0 &0 &0 \\ 
0 &1 &-1 &0 &0 &0 &0 &0 \\ 
-1 &1 &1 &0 &0 &0 &0 &0 \\ 
0 &2 &0 &-1 &-1 &-1 &-1 &0 \\ 
1 &-2 &0 &1 &1 &1 &1 &0 
\end{pmatrix}
,\\
\begin{pmatrix}
0 &0 &0 &0 &-1 &1 &0 &0 \\ 
0 &0 &1 &-1 &-1 &0 &0 &1 \\ 
0 &0 &0 &1 &1 &-1 &0 &-1 \\ 
0 &0 &0 &0 &0 &0 &1 &1 \\ 
0 &0 &-1 &0 &0 &-1 &0 &0 \\ 
1 &0 &1 &0 &0 &0 &0 &0 \\ 
0 &0 &1 &-1 &-1 &1 &0 &0 \\ 
0 &1 &-1 &1 &1 &0 &0 &0 
\end{pmatrix}
,
\begin{pmatrix}
-1 &1 &-1 &-2 &-1 &-1 &-2 &0 \\ 
0 &0 &1 &1 &0 &0 &1 &0 \\ 
-1 &0 &-1 &-1 &-1 &0 &-1 &1 \\ 
0 &-1 &0 &1 &1 &1 &2 &0 \\ 
1 &0 &2 &1 &1 &0 &1 &0 \\ 
0 &1 &0 &0 &0 &-1 &-1 &0 \\ 
0 &1 &-1 &-1 &-1 &-1 &-1 &0 \\ 
1 &0 &1 &1 &1 &1 &1 &0 
\end{pmatrix}
,
\begin{pmatrix}
1 &1 &-1 &-1 &-1 &0 &0 &1 \\ 
-1 &0 &1 &1 &1 &-1 &-1 &0 \\ 
0 &0 &-1 &0 &0 &1 &1 &0 \\ 
-1 &0 &1 &0 &0 &-1 &0 &0 \\ 
0 &-1 &1 &1 &0 &0 &0 &0 \\ 
0 &0 &0 &0 &1 &1 &0 &0 \\ 
1 &0 &-1 &-1 &0 &0 &0 &0 \\ 
1 &1 &0 &0 &0 &0 &0 &0 
\end{pmatrix}
,\\
\begin{pmatrix}
1 &1 &0 &0 &-1 &0 &0 &0 \\ 
-1 &0 &0 &-1 &1 &0 &-1 &0 \\ 
1 &-1 &0 &1 &0 &0 &0 &0 \\ 
-1 &0 &0 &-1 &0 &0 &-1 &1 \\ 
-1 &-1 &0 &0 &0 &1 &1 &0 \\ 
0 &-1 &-1 &0 &1 &0 &0 &0 \\ 
1 &1 &0 &1 &0 &0 &0 &0 \\ 
1 &1 &1 &0 &0 &0 &0 &0 
\end{pmatrix}
,
\begin{pmatrix}
0 &0 &0 &0 &0 &1 &0 &-1 \\ 
0 &0 &0 &1 &0 &1 &0 &1 \\ 
0 &0 &0 &-1 &0 &-1 &1 &0 \\ 
-1 &0 &1 &-1 &0 &-2 &0 &0 \\ 
0 &1 &-1 &0 &1 &1 &0 &0 \\ 
0 &-1 &1 &1 &0 &0 &0 &0 \\ 
0 &0 &1 &-1 &0 &-1 &0 &0 \\ 
1 &1 &-1 &1 &0 &1 &0 &0 
\end{pmatrix}
,
\begin{pmatrix}
0 &-1 &0 &0 &1 &0 &-1 &0 \\ 
0 &1 &-1 &0 &0 &0 &0 &1 \\ 
0 &-1 &1 &0 &0 &0 &1 &0 \\ 
1 &1 &-1 &0 &-1 &1 &0 &0 \\ 
-1 &-1 &1 &1 &0 &0 &0 &0 \\ 
-1 &0 &0 &0 &1 &0 &0 &0 \\ 
1 &2 &-2 &0 &0 &0 &0 &0 \\ 
0 &-1 &2 &0 &0 &0 &0 &0 
\end{pmatrix}
,\\
\begin{pmatrix}
0 &0 &0 &1 &0 &0 &0 &-1 \\ 
-1 &0 &0 &-1 &0 &0 &1 &0 \\ 
0 &0 &-1 &-1 &0 &0 &1 &0 \\ 
0 &0 &-1 &0 &0 &0 &1 &1 \\ 
0 &-1 &0 &-1 &0 &0 &0 &0 \\ 
1 &1 &1 &1 &0 &0 &0 &0 \\ 
0 &0 &1 &1 &0 &-1 &-2 &0 \\ 
0 &0 &-1 &-1 &1 &1 &1 &0 
\end{pmatrix}
.
\end{gathered}
\end{equation*}
}

\end{prop}
We can check the inequivalence of the above matrices  using the command \texttt{qfisom} from \textsf{PARI} \cite{PARI2}.
As all cells have their orientations changed, we deduce the following central result for the homology of the Voronoi complex.
\begin{thm}\label{homology_voronoiBeyond8}\label{fact3} 
 The  groups $H_k(\Vor_{\GL_N(\Z)})$ are zero modulo $\SC_2$ for $N=8$ and  $k \leq 12$, $N=9,10$ and $k<12$.
\end{thm} 

\begin{rem} In the case $N=12$ we cannot prove so far that the list of 12-dimensional cells of $\Vor_{\GL_N(\Z)}$ is complete (even if heuristically it seems the case).\end{rem}

\section{Cohomology of modular groups}\label{sec6}

\subsection{Borel--Serre duality}\label{Farrellcohomology}

According to Borel and Serre, Thm.~11.4.4 and Thm.~11.5.1 of  \cite{BS}, the group $\Gamma=\SL_N(\Z)$ or $\GL_N(\Z)$
is a virtual duality group with dualizing module
\[
H^{v(N)} (\Gamma ; \Z [\Gamma]) = {\St}_N \otimes \tilde\Z \, ,
\]
where $v(N) = N(N-1)/2$ is the virtual cohomological dimension of $\Gamma$ and $\tilde\Z$ is the orientation module
of $X_N$. It follows that there is a long exact sequence
\begin{equation}\label{lesbos}\begin{tikzcd}
& \cdots \rar & H^{v(N)-n} (\Gamma ; \tilde\Z) \rar \ar[draw=none]{d}[name=X, anchor=center]{} & \hat H^{v(N)-n} (\Gamma , \tilde\Z)
\ar[rounded corners,
to path={ -- ([xshift=2ex]\tikztostart.east)
	|- (X.center) \tikztonodes
	-| ([xshift=-2ex]\tikztotarget.west)
	-- (\tikztotarget)}]{dll} \\
& H_{n-1} (\Gamma ; {\St}_N)  \rar & H^{v(N)-n+1} (\Gamma ; \tilde\Z) \rar & \cdots.
\end{tikzcd}\end{equation}
where $\hat H^*$ is the Farrell cohomology of $\Gamma$ \cite{FA}. 


From Lemma \ref{lemma1} and the Brown spectral sequence, X (4.1) of \cite{B}, we deduce that $\hat H^* (\Gamma , \tilde\Z)$
lies in ${\mathcal S}_{N+1}$. Therefore
\begin{equation}\label{BorelSerre}
H_n (\Gamma ; \mathrm{\St}_N) \equiv  H^{v(N)-n} (\Gamma ; \tilde\Z) \, , \ \mbox{modulo ${\mathcal S}_{N+1}$.}
\end{equation}
When $N$ is odd, then ${\rm GL}_N (\Z)$ is the product of ${\rm SL}_N (\Z)$ by $\Z / 2$, therefore
\[
H^m ({\rm GL}_N (\Z) ; \Z) \equiv H^m ({\rm SL}_N (\Z) ; \Z) \, , \ \mbox{modulo ${\mathcal S}_2$.}
\]
When $N$ is even, then the action of ${\rm GL}_N (\Z)$ on $\tilde\Z$ is given by the sign of 
the determinant (see Lemma \ref{lemma2}) and Shapiro's lemma gives
\begin{equation}\label{Shapiro}
H^m ({\rm SL}_N (\Z) , \Z) = H^m ({\rm GL}_N (\Z) , M) \, ,
\end{equation}
with 
\[
M = {\rm Ind}_{{\rm SL}_N (\Z)}^{{\rm GL}_N (\Z)} \, \Z \equiv \Z \oplus \tilde \Z \, , \ \mbox{modulo ${\mathcal S}_2$.}
\]

\subsection{Homology of modular groups with coefficients in the Steinberg module}

From \ref{homology_voronoiBeyond8} and  \ref{vorsteinberg}, we deduce
\begin{cor}\label{homology_GL_steinberg}
The groups $H_{k-N+1} (\GL_N(\Z) ; {\St}_N)$ are trivial modulo $\cS_{N+1}$ for $N=8,9,10$ and $k=8,\ldots,11$ (assuming $k\ge N$).
\end{cor}

Adding this to the results of \cite{EGS}, one obtains Table \ref{VoronoiHomology}. 

\begin{table}
	\begin{tabular}{ccccccccccccc}
		\toprule
		$k\backslash N$
		& 0 & 1 & 2 & 3 & 4 & 5 & 6 & 7  & 8  & 9   & 10   &11  \\
		\midrule
		12    	& & &  &  &  &  &  & $\Z$  & ? & ?  & ?  & ? \\ 
		11   	& & &  &  &  &  &  &   & ? & ?  & ?  & ? \\ 
		10    	& & &  &  &  &  & $\Z^2$ &   & ? & ?  & ?  & ? \\ 
		9    	& & &  &  &  &  &  &   & ? & ?  & ?  & ? \\ 
		8   	& & &  &  &  &  &  &   & ? & ?  & ?  & ? \\ 
		7    	& & &  &  &  &  &  & $\Z$  & ? & ?  & ?  & ? \\ 
		6    	& & &  &  &  &  & $\Z$ & $\Z$  & ? & ?  & ?  & ? \\ 
		5    	& & &  &  &  & $\Z$ & $\Z$ &   &  & ?  & ?  & ? \\ 
		4    	& & &  &  &  &  &  &   &  & ?  & ?  & ? \\ 
		3   	& & &  & $\Z$ & $\Z$ &  &  &   &  &   & ?  & ? \\ 
		2    	& & &  &  &  &  &  &   &  &   &   & ? \\ 
		1    	& & &  &  &  &  &  &   &  &   &   & \\ 
		0      	& $\Z$ & $\Z$ &  &  &  &  &  &   &  &   &   & \\
		\bottomrule
	\end{tabular}
	\caption{The groups  $H_{k+N-1}(\Vor_{\GL_N(\Z)}) = H_k(\GL_N(\Z);\St_N)$ modulo $\cS_{N+1}$. Empty slots denote $0$.}
	\label{VoronoiHomology}
\end{table}

\subsection{Cohomology of modular groups}
When $\Gamma=\SL_N(\Z)$ or $\GL_N(\Z)$, we know $H^m(\Gamma;\tilde\Z)$ by 
combining \eqref{eq3} (end of \S\ref{ssec2.4}), Section \ref{homology_voronoi} and \eqref{BorelSerre}. As shown above, 
this allows us to compute the 
cohomology of $\Gamma$ with trivial coefficients. The results are given in Corollary \ref{cohomology} below.

\begin{cor}\label{cohomology} From Borel--Serre duality, we have
\[
 H^{\frac{N(N-1)}{2}-k} (\GL_N(\Z) ; \Z)=0 \mod \cS_{N+1}\, ,
\]
for $N=8,9,10,11$ and $0<k\leq 12-N$.
\end{cor}

\begin{rem} This provides further evidence for a conjecture of Church, Farb and Putman \cite{CFP}, see Conjecture 2.\end{rem}

\section{Application to algebraic K-theory of integers}\label{sec7}

The homology of the general linear group with coefficients in the Steinberg module can also be used to compute 
the $K$-theory of $\Z$. Let $P(\Z)$ be the exact category of free $\Z$-modules of finite rank, let $Q$  be the category obtained from $P(\Z)$ by applying Quillen's $Q$-construction \cite{QuillenI}, and let $BQ$ be its classifying space. Let $Q_N$ be the full subcategory of $Q$ containing all free $\Z$-modules of rank at most $N$ and $BQ_N$ its classifying space. 
One of the definitions of the algebraic $K$-theory groups \cite{QuillenI} is
\[
K_m (\Z) = \pi_{m+1} (BQ) \, , \quad m \geq 0 \, .
\]
Therefore we can compute $K_m (\Z)$ if we understand the homology of $BQ$ as well as the Hurewicz map
\[
h_m \colon K_m (\Z) \to H_{m+1} (BQ;\Z).
\]

To do so, we use that Quillen proved in Theorem 3 of \cite{Q3} that there are long exact sequences
\[\begin{gathered}\begin{tikzcd}
&[-5pt] \cdots \rar &[-5pt] H_m(BQ_N;\Z) \rar \ar[draw=none]{d}[name=X, anchor=center]{} &[-5pt] H_{m-N}(\GL_N(\Z);\St_N)
\ar[rounded corners,
to path={ -- ([xshift=2ex]\tikztostart.east)
	|- (X.center) \tikztonodes
	-| ([xshift=-2ex]\tikztotarget.west)
	-- (\tikztotarget)}]{dll} \\
& H_{m-1}(BQ_{N-1};\Z) \rar & H_{m-1}(BQ_N;\Z) \rar & \cdots.
\end{tikzcd}\end{gathered}\]
Since $BQ_0 \simeq \ast$, these allow us to inductively obtain information about the homology of $BQ_N$ from $H_*(\GL_N(\Z);\St_N)$.

\subsection{On the homology of $BQ$}
Using Proposition \ref{vorsteinberg}, we can rewrite the above Quillen sequences as
\begin{equation}\label{vorquillen}\begin{gathered}\begin{tikzcd}
&[-5pt] \cdots \rar &[-5pt] H_m(BQ_N;\Z) \rar \ar[draw=none]{d}[name=X, anchor=center]{} &[-5pt] H_{m-1}(\Vor_{\GL_N(\Z)})
\ar[rounded corners,
to path={ -- ([xshift=2ex]\tikztostart.east)
	|- (X.center) \tikztonodes
	-| ([xshift=-2ex]\tikztotarget.west)
	-- (\tikztotarget)}]{dll} \\
& H_{m-1}(BQ_{N-1};\Z) \rar & H_{m-1}(BQ_N;\Z) \rar & \cdots.
\end{tikzcd}\end{gathered}\end{equation}

We obtain the following result concerining the homology of $BQ$. It is proven using a spectral sequence, but a reader unfamiliar with spectral sequences may reproduce the results by inductively using the long exact sequences \eqref{vorquillen}.

\begin{prop}\label{HBQ} Modulo $\cS_7$ we have 
	\[H_m(BQ;\Z) = \begin{cases} 0 & \text{if $m=2,3,4,5,8,9$} \\
	\Z & \text{if $m=0,1,6,7,10,11$.}\end{cases}\]
Furthermore, modulo $\cS_{11}$ we have $H_{12}(BQ;\Z) = \Z$.
\end{prop}

\begin{proof}The long exact sequences \eqref{vorquillen} give rise to a spectral sequence
	\[E^1_{pq} = H_q(\GL_p(\Z);\St_p) \Longrightarrow H_{p+q}(BQ;\Z).\]
	We shall use this modulo $\cS_{7}$ and $\cS_{11}$ respectively. Table \ref{VoronoiHomology} gives most of the $E^1$-page for $p \leq 11$ and $q \leq 12$, substituting $N=p$ and $k=q$. 
	
	For $m \leq 12$, the diagonal line $p+q=m$ contains either no non-zero entry, or a single entry $\Z$. Thus to prove the first part it suffices to show that the $d^1$-differentials $d^1 \colon E^1_{p+1,q} \to E^1_{p,q}$ vanish modulo $\cS_7$ for $p+q \leq 12$. For the second part, we need to further verify that all $d^r$-differentials $d^r \colon E^r_{6+r,6-r+1} \to E^r_{6,6}$ vanish modulo $\cS_{11}$. Since these are homomorphisms into abelian groups that are either zero or free, we can prove this by verifying it after tensoring with $\Q$.
	
	To do so, we use that the rational homotopy groups of $BQ$ are known by work of Borel, \S 12 of \cite{Borel}: 
	\[\pi_m(BQ) \otimes \Q = \begin{cases} \Q & \text{if $m=1$ or, $m = 4i+2$ with $i$ a positive integer,} \\
	0 & \text{otherwise.}\end{cases}\]
Since $BQ$ is an infinite loop space, the rational homotopy groups determines its rational homology groups: for $m \leq 12$, $H_m(BQ;\Q)$ is $0$ if $m=2,3,4,5,8,9$ and $\Q$ if $m = 0,1,6,7,10,11,12$. This proves the desired statement.
\end{proof}

\begin{rem}\label{rem:cp} That the coinvariants $H_0(\GL_N(\Z);\St_N)$ vanish for $N \geq 3$ is due to Lee and Szczarba, Theorem 1.3 of \cite{LShomology}. They deduce this by exhibiting a \emph{generating set} of $\St_N$. In \cite{CP}, Church and Putman give a \emph{presentation} of $\St_N$, from which one may deduce that $H_1(\GL_N(\Z);\St_N)=0$ modulo $\cS_N$ for $N \geq 3$. In fact, Theorem A of \cite{CP} gives only a rational statement, but it is straightforward to verify their argument goes through modulo $\cS_N$. In \cite{miller2018stability}, Miller, Patzt, and Nagpal have shown that $H_{1} (\SL_N(\Z) ; \St_N)= 0$ for $N \ge 6$, that is, without needing to work modulo $\cS_N$.\end{rem}

\subsection{On the Hurewicz homomorphism}
By definition, for every integer $m \geq 1$,
\[
K_{m} (\Z) = \pi_{m+1} (BQ) \, .
\]
The space $BQ$ is in fact $\Omega^{\infty-1} \mathbf{K}(\Z)$ with $\mathbf{K}(\Z)$ the algebraic $K$-theory spectrum of $\Z$, so in particular an infinite loop space. This has consequences for the kernel $C_m$ of the Hurewicz homomorphism
\[
h_m \colon \pi_m (BQ) \to H_m (BQ;\Z) \, .
\]

\begin{prop}\label{Hurewicz}
Modulo $\cS_5$, we have $C_m = 0$ for $m=9,10,11$. Modulo $\cS_{7}$, we have $C_{12}=0$.
\end{prop}

\begin{proof}This follows from Theorem 1.5 of \cite{A}, which implies that if $X$ is a path-connected infinite loop space then the kernel of the Hurewicz homomorphism $\pi_n(X) \to H_n(X;\Z)$ is annihilated by $R_n$, an integer divisible only by primes $\leq \frac{n}{2}+1$. We apply this result to $X = BQ$.\end{proof}

\begin{thm}\label{K8Z}
 The group $K_8 (\Z)$ is trivial.
\end{thm}
\begin{proof} From Propositions \ref{HBQ} and \ref{Hurewicz} we deduce that $K_8(\Z)=0$ modulo $\cS_7$. According 
to the {\em Quillen--Lichtenbaum conjectures} (see e.g.~Chapter VI.10 of \cite{WeibelBook}) if $\ell$ is a regular odd prime, there are no $\ell$-torsion in $K_{2j}(\Z)$ for $j>0$. Hence $K_8(\Z)=0$.
\end{proof}

Using an elaboration of the method presented, we can recover information about several related algebraic K-theory groups (see also Table VI.10.1.1 of \cite{WeibelBook}).

\begin{prop}
Modulo $\cS_7$, $K_9(\Z)\iso\Z$ and $K_{10}(\Z)=0$. Modulo $\cS_{11}$, $K_{11}(\Z) = 0$.
\end{prop}

\begin{proof}Having proven Theorem \ref{K8Z}, we know the groups $K_i(\Z)$ modulo $\cS_7$ for $i \leq 8$: they vanish unless $i=0,5$ in which case they are $\Z$. The groups $K_i(\Z)$ are also the homotopy groups of the algebraic $K$-theory spectrum $\mathbf{K}(\Z)$, and we conclude that there is a map of spectra
	\[\mathbf{S}^0 \vee \mathbf{S}^1 \to \mathbf{K}(\Z)\]
which is  $8$-connected modulo $\cS_7$. By the Hurewicz theorem modulo $\cS_7$, we see that 
\[\pi_{10}(BQ) \cong H_{10}(BQ,\Omega^{\infty-1}(\mathbf{S}^0 \vee \mathbf{S}^5);\Z)\]
modulo $\cS_7$. It is a standard computation that $H_{10}(\Omega^{\infty-1}(\mathbf{S}^0 \vee \mathbf{S}^5);\Z) = 0$ modulo $\cS_7$, so from Propositions \ref{HBQ} and the long exact sequence of a pair it follows that $K_9(\Z) = \pi_{10}(BQ) \cong \Z$ modulo $\cS_7$.

This allows for the construction of a further map
	\[\mathbf{S}^0 \vee \mathbf{S}^5 \vee \mathbf{S}^9 \to \mathbf{K}(\Z)\]
which is $9$-connected modulo $\cS_7$. Applying $\Omega^{\infty-1}$ and repeating the above analysis in degrees $11$ and $12$ gives $K_{10}(\Z)=0$ modulo $\cS_7$ and $K_{11}(\Z) = 0$ modulo $\cS_{11}$.
\end{proof}

\begin{rem}As pointed out in Remark \ref{rem:cp}, \cite{CP} proves that $H_1(\GL_N(Z);\St_N)$ vanishes modulo $\cS_{N}$. In order to prove $K_{12}(\Z)=0$, we thus ''only`` need to recover the groups $H_{12}(\Vor_{\GL_N(\Z)})$ for $N=9,10,11$ which are still missing.
\end{rem}

\section{Arithmetic applications}\label{KV}\label{sec8}
For the convenience of the reader, we recall some facts about the relationship between algebraic K-theory and \'etale cohomology, with a view towards the Kummer--Vandiver conjecture. We follow the presentation of Kurihara \cite{Kurihara} and Soul\'e \cite{So2} (see also Section VI.10 of \cite{WeibelBook}). 

Let $p$ be an odd prime, $i\in\N$ and $j\in \Z$. Denote by
\[
 H^i_\et(\Z[1/p];\Z_p(j))\coloneqq \varprojlim_\nu H^i_\et(\Spec(\Z[1/p]);\Z/p^\nu(j))
\]
the \'etale cohomology groups of the scheme $\Spec(\Z[1/p])$ with coefficients in the $j$-th Tate twist
of the $p$-adic integers.
 It is known that when $j \neq 0$ these groups vanish unless $i=1,2$. 
It was shown by Dwyer and Friedlander \cite{DwyerFriedlander} and (independently) by Soul\'e \cite{So2}, that when $m=2j-i>1$ and $i=1,2$, there is a surjective Chern map
\[
 K_m(\Z) \to H^i_\et(\Z[1/p];\Z_p(j))\, .
\]
Recall the following facts about these groups:
\begin{enumerate}
 \item When $p>j+1$ the groups $H^1_\et(\Z[1/p];\Z_p(j))$ vanish.
 \item When $j>0$ is even, the order of $H^2_\et(\Z[1/p];\Z_p(j))$ is equal 
 to the numerator of $B_n/n$ (this is due to Mazur and Wiles, \cite{MazurWiles})
\end{enumerate}
Hence, those groups are not known when $i=2$ and $j$ is odd (assuming $p>j+1$). At the level of $m$ 
it means that {$m$ is divisible by 4}.
Let $\Q(\zeta_p)$ be the cyclotomic extension of $\Q$ obtained by adding $p$-th roots of unity.
Let $C$ be the $p$-Sylow subgroup of the class group of $\Q(\zeta_p)$.
The group $\Delta=\Gal(\Q(\zeta_p)/\Q)\iso \U{(\Z/p)}$ acts upon $C$ via the Teichm\"uller character
\[ \omega \colon \Delta \to \U{(\Z/p)}\, ,\]
with $g(x)=x^{\omega(g)}$ and $x^p=1$.
For all $i\in \Z$ let
\[ C^{(i)} = \{ x \in C \quad\text{such that}\quad g(x)=\omega(g)^i x \quad \text{for all}\quad g\in \Delta\}\, . \] 
Let  $C^+$ be the subgroup of $C$ fixed by the complex conjugation of $\Q(\zeta_p)$. The Kummer-Vandiver conjecture states that $C^+=0$ for arbitrary $p$. By the above construction, it turns out that $C^+$
is the direct sum of the groups $C^{(i)}$ for $i$ even and $0\leq i \leq p-3$. We then deduce the reformulation 
of the Kummer--Vandiver conjecture \cite{Kurihara,So2}:

\begin{conj}[Kummer--Vandiver conjecture] The groups $C^{(i)}$ vanish for $i$ even and $0\leq i \leq p-3$.\end{conj}

From \emph{op. cit.}, using the above setting, we get a surjective map
\[
 K_{2m-2}(\Z) \twoheadrightarrow C^{(p-m)}\, .
\]
As consequence of (\ref{K8Z}), we get
\begin{cor} The groups  $C^{(p-5)}$ are zero for all prime $p>3$.
\end{cor}
\begin{rem} By \cite{So2}, we know that $C^{(p-n)}$ is zero for $p$ ``large enough'' with respect to $n$.
From computations done by Buhler and  Harvey \cite{BH11}, we know that the conjecture is true for all (irregular) primes $p<163 577 856$. This was recently improved to $p<2 147 483 648$ by Hart, Harvey, and Ong \cite{HartHarveyOng}.
\end{rem}

\bibliographystyle{plain}
\bibliography{References}

\begin{thebibliography}{10}

\bibitem{A}
D.~Arlettaz.
\newblock The {H}urewicz homomorphism in algebraic {$K$}-theory.
\newblock {\em J. Pure Appl. Algebra}, 71(1):1--12, 1991.

\bibitem{Bacher17}
R.~Bacher.
\newblock On the number of perfect lattices.
\newblock {\em J. Th\'{e}or. Nombres Bordeaux}, 30(3):917--945, 2018.

\bibitem{barnes}
E.~S. Barnes.
\newblock The complete enumeration of extreme senary forms.
\newblock {\em Philos. Trans. Roy. Soc. London. Ser. A.}, 249:461--506, 1957.

\bibitem{BergeMartinet}
A.-M. Berg\'e and J.~Martinet.
\newblock On perfection relations in lattices.
\newblock {\em Contemporary Math.}, 493:29--49, 2009.

\bibitem{Borel}
A.~Borel.
\newblock Stable real cohomology of arithmetic groups.
\newblock {\em Ann. Sci. \'Ecole Norm. Sup. (4)}, 7:235--272 (1975), 1974.

\bibitem{BS}
A.~Borel and J.-P. Serre.
\newblock Corners and arithmetic groups.
\newblock {\em Comment. Math. Helv.}, 48:436--491, 1973.
\newblock Avec un appendice: Arrondissement des vari\'et\'es \`a coins, par A.
  Douady et L. H\'erault.

\bibitem{MAGMA}
W.~Bosma, J.~Cannon, and C.~Playoust.
\newblock The {M}agma algebra system. {I}. {T}he user language.
\newblock {\em J. Symbolic Comput.}, 24(3-4):235--265, 1997.
\newblock Computational algebra and number theory (London, 1993).

\bibitem{B}
K.~S. Brown.
\newblock {\em Cohomology of groups}, volume~87 of {\em Graduate Texts in
  Mathematics}.
\newblock Springer-Verlag, New York, 1994.
\newblock Corrected reprint of the 1982 original.

\bibitem{BH11}
J.~P. Buhler and D.~Harvey.
\newblock Irregular primes to 163 million.
\newblock {\em Mathematics of Computation}, 80(276):2435--2435, 2011.

\bibitem{CFP}
T.~Church, B.~Farb, and A.~Putman.
\newblock A stability conjecture for the unstable cohomology of {$SL_n(Z)$},
  mapping class groups, and {$Aut(F_n)$}.
\newblock {\em Contemporary Mathematics}, pages 55--70, 2014.

\bibitem{CP}
T.~Church and A.~Putman.
\newblock The codimension-one cohomology of {$SL_n(Z)$}.
\newblock {\em Geometry \& Topology}, 21(2):999--1032, Mar 2017.

\bibitem{cohenbook}
H.~Cohen.
\newblock {\em A course in computational algebraic number theory}, volume 138
  of {\em Graduate Texts in Mathematics}.
\newblock Springer-Verlag, Berlin, 1993.

\bibitem{smoothness}
M.~Dutour~Sikiri\'c, K.~Hulek, and A.~Sch\"urmann.
\newblock Smoothness and singularities of the perfect form and the second
  {V}oronoi compactification of {${\mathcal A}_g$}.
\newblock {\em Algebr. Geom.}, 2(5):642--653, 2015.

\bibitem{dsv}
M.~Dutour~Sikiri\'c, A.~Sch\"urmann, and F.~Vallentin.
\newblock Classification of eight-dimensional perfect forms.
\newblock {\em Electron. Res. Announc. Amer. Math. Soc.}, 13:21--32, 2007.

\bibitem{DwyerFriedlander}
W.~G. Dwyer and E.~M. Friedlander.
\newblock Algebraic and etale {$K$}-theory.
\newblock {\em Trans. Amer. Math. Soc.}, 292(1):247--280, 1985.

\bibitem{pev16}
P.~Elbaz-Vincent.
\newblock Perfect forms of rank $\leq 8$, triviality of ${K}_8(\mathbf{{Z}})$
  and the {K}ummer/{V}andiver conjecture.
\newblock In Renaud Coulangeon, Benedict~H. Gross, and Gabriele Nebe, editors,
  {\em Lattices and Applications in Number Theory}, volume~13 of {\em 1}, 2016.

\bibitem{EGS}
P.~Elbaz-Vincent, H.~Gangl, and C.~Soul\'e.
\newblock Perfect forms, {K}-theory and the cohomology of modular groups.
\newblock {\em Adv. Math.}, 245:587--624, 2013.

\bibitem{FA}
F.~T. Farrell.
\newblock An extension of {T}ate cohomology to a class of infinite groups.
\newblock {\em J. Pure Appl. Algebra}, 10(2):153--161, 1977/78.

\bibitem{mpfr}
L.~Fousse, G.~Hanrot, V.~Lef\`evre, P.~P\'elissier, and P.~Zimmermann.
\newblock M{PFR}: a multiple-precision binary floating-point library with
  correct rounding.
\newblock {\em ACM Trans. Math. Software}, 33(2):Art. 13, 15, 2007.

\bibitem{cdd}
K.~Fukuda.
\newblock cdd.
\newblock https://www.inf.ethz.ch/personal/fukudak/cdd\_home, 2016.

\bibitem{gauss}
C.-F. Gauss.
\newblock Untersuchungen \"uber die eigenschaften der positiven tern\"aren
  quadratischen formen von ludwig august seeber.
\newblock {\em J. Reine Angew. Math.}, 20:312--320, 1840.

\bibitem{glpk}
The~GLPK group.
\newblock glpk.
\newblock https://www.gnu.org/software/glpk, 2017.

\bibitem{PARI2}
The~PARI group.
\newblock {\em PARI/GP, Versions {\tt 2.1 -- 2.4}}.

\bibitem{StableBetti}
S.~Grushevsky, K.~Hulek, and O.~Tommasi.
\newblock Stable betti numbers of (partial) toroidal compactifications of the
  moduli space of abelian varieties.
\newblock In {\em Proceedings in honour of Nigel Hitchin's 70th birthday,
  Volume II}, pages 581--610. Oxford University Press, 2018.
\newblock {With an appendix by Mathieu Dutour Sikiri{\'c}}.

\bibitem{eigenweb}
G.~Guennebaud, B.~Jacob, et~al.
\newblock Eigen v3.
\newblock http://eigen.tuxfamily.org, 2010.

\bibitem{HartHarveyOng}
W.~Hart, D.~Harvey, and W.~Ong.
\newblock Irregular primes to two billion.
\newblock {\em Math. Comp.}, 86(308):3031--3049, 2017.

\bibitem{jaquet}
D.-O. Jaquet-Chiffelle.
\newblock \'enum\'eration compl\`ete des classes de formes parfaites en
  dimension {$7$}.
\newblock {\em Ann. Inst. Fourier (Grenoble)}, 43(1):21--55, 1993.

\bibitem{K-M-S}
W.~Keller, J.~Martinet, and A.~Sch\"urmann.
\newblock On classifying {M}inkowskian sublattices.
\newblock {\em Math. Comp.}, 81(278):1063--1092, 2012.
\newblock With an appendix by Mathieu Dutour Sikiri\'c.

\bibitem{zolotarev72}
A.~N. Korkin and E.~I. Zolotarev.
\newblock Sur les formes quadratiques positives quaternaires.
\newblock {\em Math. Ann.}, 5(1):581--583, 1872.

\bibitem{zolotarev77}
A.~N. Korkin and E.~I. Zolotarev.
\newblock Sur les formes quadratiques positives.
\newblock {\em Math. Ann.}, 11(1):242--292, 1877.

\bibitem{kupersshort}
A.~Kupers.
\newblock A short proof that {$K_8(Z) \cong 0$}.
\newblock 2017.

\bibitem{Kurihara}
M.~Kurihara.
\newblock Some remarks on conjectures about cyclotomic fields and {$K$}-groups
  of {$\Z$}.
\newblock {\em Compositio Math.}, 81(2):223--236, 1992.

\bibitem{KPCR}
J.~Kuzmanovich and A.~Pavlichenkov.
\newblock Finite groups of matrices whose entries are integers.
\newblock {\em Amer. Math. Monthly}, 109(2):173--186, 2002.

\bibitem{lagrange}
J.-L. Lagrange.
\newblock D\'emonstration d'un th\'eor\`eme d'arithm\'etique.
\newblock {\em Nouv. M\'em. Acad. Berlin, in Oeuvre de Lagrange III},
  20:189--201, 1770.

\bibitem{LShomology}
R.~Lee and R.~H. Szczarba.
\newblock On the homology and cohomology of congruences subgroups.
\newblock {\em Inventiones Math.}, 33:15--53, 1976.

\bibitem{LS}
R.~Lee and R.~H. Szczarba.
\newblock On the torsion in {$K_{4}(\Z)$} and {$K_{5}(\Z)$}.
\newblock {\em Duke Math. J.}, 45(1):101--129, 1978.

\bibitem{M1}
J.~Martinet.
\newblock Sur l'indice d'un sous-r\'eseau.
\newblock In {\em R\'eseaux euclidiens, designs sph\'eriques et formes
  modulaires}, volume~37 of {\em Monogr. Enseign. Math.}, pages 163--211.
  Enseignement Math., Geneva, 2001.
\newblock With an appendix by Christian Batut.

\bibitem{martinet}
J.~Martinet.
\newblock {\em Perfect lattices in {E}uclidean spaces}, volume 327 of {\em
  Grundlehren der Mathematischen Wissenschaften [Fundamental Principles of
  Mathematical Sciences]}.
\newblock Springer-Verlag, Berlin, 2003.

\bibitem{MazurWiles}
B.~Mazur and A.~Wiles.
\newblock Class fields of abelian extensions of {$\Q$}.
\newblock {\em Invent. Math.}, 76(2):179--330, 1984.

\bibitem{miller2018stability}
J.~Miller, R.~Nagpal, and P.~Patzt.
\newblock Stability in the high-dimensional cohomology of congruence subgroups,
  2018.
\newblock Preprint at {\tt arXiv:1806.11131}.

\bibitem{Q3}
D.~Quillen.
\newblock Finite generation of the groups $k_i$ of rings of algebraic integers.
\newblock {\em Springer Lecture Notes in Mathematics}, 341:179--198, 1973.

\bibitem{QuillenI}
D.~Quillen.
\newblock Higher algebraic {$K$}-theory. {I}.
\newblock pages 85--147. Lecture Notes in Math., Vol. 341, 1973.

\bibitem{schrijver}
A.~Schrijver.
\newblock {\em Theory of linear and integer programming}.
\newblock Wiley-Interscience Series in Discrete Mathematics. John Wiley \&
  Sons, Ltd., Chichester, 1986.
\newblock A Wiley-Interscience Publication.

\bibitem{Sch}
A.~Sch\"urmann.
\newblock Enumerating perfect forms.
\newblock In {\em Quadratic forms---algebra, arithmetic, and geometry}, volume
  493 of {\em Contemp. Math.}, pages 359--377. Amer. Math. Soc., Providence,
  RI, 2009.

\bibitem{Soule-SL3}
C.~Soul\'e.
\newblock The cohomology of {${\rm SL}_{3}(\Z)$}.
\newblock {\em Topology}, 17(1):1--22, 1978.

\bibitem{So2}
C.~Soul\'e.
\newblock Perfect forms and the {V}andiver conjecture.
\newblock {\em J. Reine Angew. Math.}, 517:209--221, 1999.

\bibitem{SouleK4}
C.~Soul\'e.
\newblock On the {$3$}-torsion in {$K_4(\Z)$}.
\newblock {\em Topology}, 39(2):259--265, 2000.

\bibitem{woerdenMasterThesis}
W.~P.~J. van Woerden.
\newblock Perfect quadratic forms: An upper bound and challenges in
  enumeration.
\newblock Master thesis of Leiden University, 2018.

\bibitem{woerden2019upper}
W.~P.~J. van Woerden.
\newblock An upper bound on the number of perfect quadratic forms, 2019.
\newblock Preprint at {\tt arXiv:1901.04807}.

\bibitem{Vo}
G.~Voronoi.
\newblock Nouvelles applications des param\`etres continus \`a la th\'eorie des
  formes quadratiques 1: Sur quelques propri\'et\'es des formes quadratiques
  positives parfaites.
\newblock {\em J. Reine Angew. Math}, 133(1):97--178, 1908.

\bibitem{WeibelBook}
C.~A. Weibel.
\newblock {\em The {$K$}-book}, volume 145 of {\em Graduate Studies in
  Mathematics}.
\newblock American Mathematical Society, Providence, RI, 2013.
\newblock An introduction to algebraic $K$-theory.

\end{thebibliography}

\end{document}